\documentclass[11pt]{amsart}
\newcommand{\ben}{\begin{enumerate}}
\newcommand{\een}{\end{enumerate}}
\newcommand{\eq}[2][label]{\begin{equation}\label{#1}#2\end{equation}}
\newcommand{\av}[2]{\langle #1\rangle_{_{\scriptstyle #2}}}

\newcommand{\avm}[2]{\langle #1\rangle_{_{\scriptstyle #2,\mu}}}

\usepackage{amsfonts}
\usepackage{amssymb}
\usepackage{amsmath}
\usepackage{graphicx,xcolor,hyperref,enumitem}
\newcommand{\cbl}[1]{{\color{black}#1}}

\newcommand{\cma}[1]{{\color{black}#1}}

\newcommand{\bel}[1]{\boldsymbol{#1}}

\newcommand{\ma}{Monge--Amp\`{e}re }

\newcommand{\BMO}{{\rm BMO}}
\newcommand{\BLO}{{\rm BLO}}

\newcommand{\T}{\mathcal{T}}

\newcommand{\rn}{\mathbb{R}^n}

\newcommand{\vf}{\varphi}

\newtheorem{theorem}{Theorem}[section]

\newtheorem{lemma}[theorem]{Lemma}

\newtheorem*{theorem*}{Theorem}{\bf}{\it}
\newtheorem*{proposition*}{Proposition}{\bf}{\it}
\newtheorem*{observation*}{Observation}{\bf}{\it}
\newtheorem*{lemma*}{Lemma}{\bf}{\it}

\theoremstyle{definition}
\newtheorem{definition}[theorem]{Definition}

\theoremstyle{remark}
\newtheorem{remark}[theorem]{Remark}

\numberwithin{equation}{section}

\allowdisplaybreaks

\begin{document}

\title{The $\!\BMO\to\BLO\!$ action of the maximal operator on $\alpha$-trees}

\author{Adam Os\c{e}kowski}
\address{Faculty of Mathematics, Informatics and Mechanics, University of Warsaw,
Banacha 2, 02-097 Warsaw, Poland}
\email{ados@mimuw.edu.pl}

\author{Leonid Slavin}
\address{University of Cincinnati}
\email{leonid.slavin@uc.edu}

\author{Vasily Vasyunin}
\address{St. Petersburg Department of the V.~A.~Steklov
Mathematical Institute, RAS, and St. Petersburg State University}
\email{vasyunin@pdmi.ras.ru}

\thanks{L. Slavin's and V. Vasyunin's research was supported by the Russian Science Foundation grant 14-41-00010}



\begin{abstract}
We obtain the explicit upper Bellman function for the natural dyadic maximal operator acting from $\BMO(\rn)$ into $\BLO(\rn).$ As a consequence, we show that the $\BMO\to\BLO$ norm of the natural operator equals 1 for all $n,$ and so does the norm of the classical dyadic maximal operator. 
The main result is a partial corollary of a theorem for \cma{the so-called} $\alpha$-trees, which generalize dyadic lattices. The Bellman function in this setting exhibits an interesting quasi-periodic structure depending on $\alpha,$ but also allows a majorant independent of $\alpha,$ hence the dimension-free norm constant. \cma{We also describe the decay of the norm with respect to the} difference between the average of a function on a cube and the infimum of \cma{its} maximal function on that cube. An explicit norm-\cma{optimizing} sequence is constructed.
\end{abstract}
\maketitle

\section{Introduction and main results}
We are interested in the action of the maximal operator on BMO. In~\cite{bds}, Bennett, DeVore, and Sharpley showed that the Hardy-Littlewood maximal function maps $\BMO$ to itself. In~\cite{bennett}, Bennett strengthened this result by showing that it actually maps $\BMO$ to a subclass of $\BMO$ called $\BLO$ (``bounded lower oscillation''). Bennett's proof is elementary, but the estimates it gives are not sharp. \cma{As far as we know}, the exact \cma{operator norm} in this setting has not been evaluated for any maximal \cma{operator}. In this paper, we conduct a detailed study of the action of the dyadic maximal operator, as well as more general maximal operators on trees, from BMO into BLO. Let us first set forth the necessary definitions.

We will use $\mathcal{D}$ to denote the collection of all open dyadic cubes in $\rn.$ If a cube $Q$ is fixed, then $\mathcal{D}(Q)$ denotes the collection of all dyadic subcubes of $Q.$
The symbol $\av{\varphi}J$ will stand for the average of a locally integrable function over a set $J$ with respect to the Lebesgue measure; if a different measure, $\mu,$ is involved, we write
$\av{\varphi}{J,\mu}.$ Thus,
$$
\av{\varphi}J=\frac1{|J|}\int_J\varphi,\qquad \avm{\varphi}{J}=\frac1{\mu(J)}\int_J\varphi\,d\mu.
$$
 
The dyadic $\BMO$ on $\rn$ is defined as follows:
\eq[1]{
\BMO^d(\rn)=\big\{\varphi\in L^2_{loc}\colon\|\varphi\|_{\BMO^d}:=\sup_{J\in\mathcal{D}}\big(\av{\varphi^2}J-\av{\varphi}J^2\big)^{1/2}<\infty\big\}.
}
We will also use $\BMO^d(Q)$ when the supremum is taken over all $J\in\mathcal{D}(Q)$ for some cube $Q.$

The dyadic $\BLO$ on $\rn$ is defined by:
\eq[1.1]{
\BLO^d(\rn)=\big\{\varphi\in L^1_{loc}\colon\|\varphi\|_{\BLO^d}:=\sup_{J\in\mathcal{D}}\big(\av{\varphi}J-\inf_J\varphi\big)<\infty\big\}.
}
\cma{(Throughout the paper, we use ``$\inf$'' as shorthand for ``${\rm ess\, inf}$''.)} BLO was introduced by Coifman and Rochberg in~\cite{cr}. It is easy to see that $\BLO\subset\BMO.$ However, this inclusion is proper: for instance, the function $t\mapsto \log|t|$ is in $\BMO(\rn),$ but not in $\BLO(\rn).$ (This also shows that the class BLO is not linear and, in particular, not preserved under multiplication by a constant, as the function $-\log |t|$ is in BLO.) A useful viewpoint is this: per the John--Nirenberg inequality, a BMO function is a constant multiple of the logarithm of an $A_\infty$ weight; on the other hand, as shown in~\cite{cr}, a BLO function is a non-negative multiple of the logarithm of an $A_1$ weight.

We consider two dyadic maximal operators. The first one is the classical dyadic maximal function given by 
$$
M\varphi(x)=\sup_{\cma{J\ni x; J\in\mathcal{D}}}\av{|\varphi|}J.
$$
The second is the so-called natural counterpart of $M,$ without the absolute value in the average:
$$
N\varphi(x)=\sup_{\cma{J\ni x; J\in\mathcal{D}}}\av{\varphi}J.
$$
Obviously, $M$ and $N$ coincide on non-negative functions.

In light of Bennett's result from~\cite{bennett}, we expect that both $M$ and $N$ would map $\BMO^d$ to $\BLO^d,$ which can be written as follows: for any $Q\in\mathcal{D},$
$$
\av{M\varphi}Q\le c_n\|\varphi\|_{\BMO^d(\rn)}+\inf_QM\varphi,
\qquad \av{N\varphi}Q\le c_n\|\varphi\|_{\BMO^d(\rn)}+\inf_QN\varphi.
$$
We will first show \cma{the} inequality for the operator $N,$ with the sharp constant $c_n$ and with the BMO norm taken over $Q,$ as opposed to all of $\rn.$ The inequality for $M,$ with the same constant, then follows easily:
$$
\av{M\varphi}Q-\inf_QM\varphi=\av{N|\varphi|}Q-\inf_QN|\varphi|\le c_n\||\varphi|\|_{\BMO^d(Q)}\le c_n\|\varphi\|_{\BMO^d(Q)},
$$
where the last inequality follows because $\av{\varphi^2}J-\av{|\varphi|}J^2\le \av{\varphi^2}J-\av{\varphi}J^2$  for any $J.$ In Section~\ref{optimizers} we will provide a non-negative optimizing sequence, proving the sharpness of both inequalities. 

The norm inequality for $N$ is, in fact, a special case of a more general inequality relating $\av{N\varphi}Q,$ $\inf_Q N\varphi,$ $\av\varphi Q,$ and $\|\varphi\|_{\BMO^d(Q)}.$ 
Here is our main theorem.
\begin{theorem}
\label{main_th}
Let $Q\in\mathcal{D}.$ Take a function $\varphi\in L^1_{loc}(\rn)$ such that $N\varphi$ is not identically infinite on $Q$ and $\varphi|_Q\in\BMO^d(Q).$  Let $L=\inf_QN\varphi;$ $t=\inf_QN\varphi-\av{\varphi}Q.$ Then
\eq[mt1]{
\av{N\varphi}Q\le L+\Phi_n\big(t\big)\, \|\varphi\|_{\BMO^d(Q)},
}
\cma{where $\Phi_n$ is a decreasing\textup, convex function on $[0,\infty)$ satisfying
\eq[mt1.5]{
\Phi_n(k(2^{n/2}-2^{-n/2}))=2^{-nk}
}
for all non-negative integers $k.$} Consequently\textup, if $\varphi\in\BMO^d(\rn),$ then $N\varphi\in\BLO$ and
\eq[mt2]{
\|N\varphi\|_{\BLO^d(\rn)}\le \|\varphi\|_{\BMO^d(\rn)}.
}
Both inequalities are sharp. Moreover, both inequalities remain true if the operator $N$ is replaced with the operator $M,$ and inequality~\eqref{mt2} remains sharp under this replacement.
\end{theorem}

\begin{remark}
In Section~\ref{bellman_formulation}, we give the optimal function $\Phi_n(t)$ for all $t$ (specifically, $\Phi_n(t)=\cma{b(-t)}$ where $b$ is given by~\eqref{bp}--\eqref{gy} for $\alpha=2^{-n}$); it has too complicated an expression to be useful in the context of this theorem. 
\cma{For} that $\Phi_n,$ the stated sharpness of~\eqref{mt1} means that for any dyadic cube $Q,$ any real number $L,$ and any $t\ge0,$ there exists a sequence of functions $\{\varphi_j\}$ such that $\varphi_j|_Q\in\BMO^d(Q),$ with $\|\varphi_j\|_{\BMO^d(Q)}=1;$ $\inf_QN\varphi_j=L;$ and $\inf_QN\varphi_j-\av{\varphi_j}Q=t;$ and also such that
$$
\lim_{j\to\infty}\av{N\varphi_j}Q=L+\Phi_n\big(t\big).
$$ 
Likewise, the sharpness of~\eqref{mt2} means that there is a sequence $\{\varphi_j\}$ of functions from $\BMO^d(\rn)$ such that $\|\varphi_j\|_{\BMO^d(\rn)}=1$ for all $j$ and
$$
\|N\varphi_j\|_{\BLO^d(\rn)}\to1~\text{as}~j\to\infty.
$$
\end{remark}

The theorem for dyadic BMO stated above is a partial corollary of a more general theorem for special structures that we will call $\alpha$-trees. These structures generalize dyadic lattices. \cma{They} were introduced by the second and third authors in~\cite{alpha_trees} to obtain the sharp John--Nirenberg inequality for $\BMO^d(\rn).$ However, the first author had previously used very similar tree structures in~\cite{ose1} (called $\alpha$-splitting trees there) to obtain sharp weak-type maximal inequalities for dyadic $A_1(\rn).$ The proofs in~\cite{alpha_trees},~\cite{ose1}, as well as in the current paper rely on Bellman functions adapted to trees. Similar nested structures have also been used in Bellman-function contexts by Melas \cite{melas} and Melas, Nikolidakis, and Stavropoulos \cite{mns}; see also the works of Ba\~nuelos and Os\c{e}kowski~\cite{bose} and~Os\c{e}kowski \cite{ose2}. The important distinction is that the trees used by those authors were homogeneous, meaning every element of the tree had the same number of offspring, all of the same measure. In our definition (and in the one in~\cite{ose1}), the number of offspring is not restricted, as long as none is too small \cma{relative to the parent}.

\begin{definition}
\label{tree}
Let $(X,\mu)$ be a measure space with $0<\mu(X)<\infty.$ Let $\alpha\in(0,1/2].$ A collection $\T$ of measurable subsets of $X$ is called an $\alpha$-tree, if the following conditions are satisfied:
\ben
\item
$X\in\T.$
\item
For every $J\in\T,$ there exists a subset $C(J)\subset\T$ such that 
\ben
\item
$J=\bigcup_{I\in C(J)} I,$
\item
the elements of $C(J)$ are pairwise disjoint up to sets of measure zero,
\item
for any $I\in C(J),$ $\mu(I)\ge\alpha\mu(J).$
\een
\item
$\T=\bigcup_m\T_m,$ where $\T_0=\{X\}$ and $\T_{m+1}=\bigcup_{J\in\T_m}C(J).$ 
\item
The family $\T$ differentiates $L^1(X,\mu)$: \cma{for each $x\in X,$ let $J^x_k$ be any element of $\T_k$ containing $x.$ Then for $\mu$-almost every $x\in X$ and every $f\in L^1(X,\mu),$ we have
$
\lim_{k\to\infty}\avm{f}{J_k^x}=f(x).
$
}
\een
 

Observe that each $C(J)$ is necessarily finite. We will refer to the elements of $C(J)$ as children of $J$ and to $J$ as their parent. Also note that $\T(J):=\{I\in\T: I\subset J\}$ is an $\alpha$-tree on $(J,\mu|_J).$ We write $\mathcal{T}_k(J)$ for the collection of all descendants of $J$ of the $k$-th generation relative to $J;$ thus, $\T(J)=\bigcup_k\T_k(J).$

\end{definition}

If $\alpha\in(0,1/2]$ and $\T$ is an $\alpha$-tree on a measure space $(X,\mu),$ then we can define the associated $\BMO,$ $\BLO,$ and maximal operators as follows:
$$
\varphi\in\BMO(\mathcal{T})\Longleftrightarrow \|\varphi\|_{\BMO(\mathcal{T})}:=\sup_{J\in\mathcal{T}}\{\av{\varphi^2}{J,\mu}-\av{\varphi}{J,\mu}^2\}^{1/2}<\infty,
$$
$$
\varphi\in\BLO(\mathcal{T})\Longleftrightarrow \|\varphi\|_{\BLO(\mathcal{T})}:=\sup_{J\in\mathcal{T}}\{\av{\varphi}{J,\mu}-\inf_J\varphi\}<\infty,
$$
$$
M_\T\varphi(x)=\sup_{\cma{J\ni x; J\in\T}}\av{|\varphi|}J,\qquad N_\T\varphi(x)=\sup_{\cma{J\ni x; J\in\T}}\av{\varphi}J.
$$
With these definitions, we have the following theorem.
\begin{theorem}
\label{main_th_trees}
Let $(X,\mu)$ be a measure space, $\alpha\in(0,1/2],$ $\T$ be an $\alpha$-tree on $X,$ and $K\in\T.$ Take a function $\varphi$ on $X$ such that $N_\T\varphi$ is not identically infinite on $K$ and $\varphi|_K\in\BMO\big(\T(K)\big).$  Let $L=\inf_KN_\T\varphi$ \cma{and} $t=\inf_KN_\T\varphi-\avm{\varphi}K.$ Then
\eq[mt1.1]{
\avm{N_\cma{\T}\varphi}K\le L+\mathcal{F}_\alpha\big(t\big)\, \|\varphi\|_{\BMO(\T(K))},
}
where $\mathcal{F}_\alpha$ is a \cma{decreasing convex function on $[0,\infty)$ satisfying
\eq[mt1.6]{
\mathcal{F}_\alpha\Big(k\big(\alpha^{-1/2}-\alpha^{1/2}\big)\Big)=\alpha^k
}
for all non-negative integers $k.$} Consequently\textup, $N_\T\varphi\in\BLO(\T)$ and
\eq[mt2.1]{
\|N_\T\varphi\|_{\BLO(\T)}\le \|\varphi\|_{\BMO(\T)}.
}
Both inequalities remain true if the operator $N_\T$ is replaced with the operator $M_\T.$
\end{theorem}

\begin{remark}
Setting \cma{in this theorem} $\alpha=2^{-n}$ and $\mathcal{F}_\alpha=\Phi_n,$ we immediately obtain the inequalities in Theorem~\ref{main_th}. However, we do not claim sharpness here, thus the sharpness in the dyadic case must be established separately. If one restricts consideration to non-atomic trees as done in the work of Melas and co-authors, e.g. in \cite{melas,melas1,mn,mns},  and demands that each element $K$ of a tree have a child of measure $\alpha\mu(K),$ then for each tree one can construct optimizing sequences for the inequalities in Theorem~\ref{main_th_trees}. With our definition, however, one can easily come up with a tree for which the inequalities will not be sharp.
\end{remark}

We close this section with a brief history of our project. It grew out of a related project, one devoted to sharp estimates for the dyadic maximal operator acting on $A_\infty.$ As shown in~\cite{ou} (see also \cite{ou1}), \cma{the boundedness of the operator} $N\colon \BMO\to\BLO$ \cma{is equivalent to the boundedness of the operator} $M\colon A_\infty\to A_1$  (the proof in~\cite{ou} is given for the non-dyadic case, but it works for any maximal \cma{operator}). However, the sharpness in \cma{one of the corresponding} inequalities \cma{does not transfer to the other}, so to get the sharp bounds one needs to deal with these questions separately. It turns out that the $N\colon \BMO\to\BLO$ question considered here is computationally easier and thus makes for a better starting point. The $M\colon A_\infty\to A_1$ question will be considered elsewhere.

The rest of the paper is organized \cma{as} follows.
In Section~\ref{bellman_formulation}, we define the Bellman function for the dyadic problem. We also define $\alpha$-concave functions and show that \cma{a suitable} family of such functions would provide a majorant for the left-hand side of~\eqref{mt1.1} and thus also majorate the dyadic Bellman function. We then give an explicit formula for such a family function, but postpone the (somewhat technical) verification of its $\alpha$-concavity until Section~\ref{alpha_concavity}. Subject to that verification, this establishes the upper estimates in Theorems~\ref{main_th} and~\ref{main_th_trees}. In Section~\ref{optimizers}, we show that the dyadic majorant is, in fact, equal to the dyadic Bellman function; in particular, this establishes the sharpness of the inequalities in~Theorem~\ref{main_th}. Our proof uses abstract concavity properties of the Bellman function and does not rely on explicit optimizers. However, we also present a norm-optimizing sequence for the operator $N.$
In Section~\ref{alpha_concavity}, we verify the $\alpha$-concavity assumed in Section~\ref{bellman_formulation}. Finally, in Section~\ref{how_to}, we outline how we obtained the Bellman candidate presented in Section~\ref{bellman_formulation}.

\section{The Bellman function, $\alpha$-concavity, and the main Bellman theorem}
\label{bellman_formulation}

To prove Theorem~\ref{main_th}, we compute the corresponding Bellman function, which is the solution of the underlying extremal problem. To define this function, first consider the following parabolic domain in the plane:
$$
\Omega=\{(x_1,x_2)\colon x_1^2\le x_2\le x_1^2+1\}.
$$
By $\Gamma_0$ and $\Gamma_1$ we denote the lower and upper boundaries of $\Omega,$ respectively:
$$
\Gamma_0=\{(x_1,x_2)\colon x_2=x_1^2\},\qquad
\Gamma_1=\{(x_1,x_2)\colon x_2=x_1^2+1\}.
$$
The domain of our Bellman function will be the following set in $\mathbb{R}^3:$
$$
S=\{(x_1,x_2,L)\colon (x_1,x_2)\in\Omega, ~x_1\le L\}.
$$
We will often write $(x_1,x_2,L)$ as $(x,L)$ with $x\in\mathbb{R}^2.$ 

For every $(x,L)\in S$ and every $Q\in\mathcal{D}$ we designate a special subset of functions on $\rn$ whose restrictions to $Q$ are in $\BMO^d(Q);$ we will refer to its elements as the test functions:
\eq[b_test]{
\begin{aligned}
E_{x,L,Q}=\big\{\vf\in L^1_{loc}(\rn),&~\vf|_Q\in\BMO^d(Q),~ \|\vf\|_{\BMO^d(Q)}\le1,\\
&\av{\vf}Q=x_1,~\av{\vf^2}Q=x_2, \sup_{\cma{R\supset Q; R\in\mathcal{D}}} \av{\varphi}R=L\big\}.
\end{aligned}
}
It is an easy exercise to show that the set $E_{x,L,Q}$ is non-empty for any $Q\in\mathcal D$ and any $(x,L)\in S;$ in fact, one can construct an appropriate test function that would take at most two values on $Q$ and at most one other value on $\rn\setminus Q.$

Now we define the following {\it Bellman function} on $S$:
\eq[b_main]{
\bel{B}^n(x,L)=\sup \{\av{N\vf}Q\colon \varphi\in E_{x,L,Q}\}.
}
Various properties of this function can be derived directly from the definition, and we will do so below, in Section~\ref{properties}. (One immediate observation, by simple rescaling, is that $\bel{B}^n$ does not actually depend on the cube $Q.$) Definition~\eqref{b_test}--\eqref{b_main} combines two well-known Bellman formulations: the one for the dyadic maximal operator on $L^2$ given in~\cite{nt}, and the one for a general integral functional on $\BMO$ with the square norm, first given in~\cite{sv1} and then fully developed in~\cite{iosvz2} and subsequent work.

The Bellman function defined by Nazarov and Treil in~\cite{nt} was not computed in that paper. That was done by Melas in~\cite{melas} (for all $L^p,$ $p>1$). Melas's computation relied on a careful analysis of combinatorial properties of the operator. A different, PDE-based approach was implemented in~\cite{ssv}; it is also the one we employ here (see Section~\ref{how_to} for details).

The reader may be wondering why in defining the Bellman function for the $\BMO\to\BLO$ action of the operator $N$ we did not fix $\inf_Q N\varphi.$ The following simple observation provides the answer.
\begin{lemma}
\label{key_obs}
Take $\alpha\in(0,\frac12], $ and let $\T$ be an $\alpha$-tree on a measure space $(X,\mu).$ Take any $\varphi\in L^1(X)$ and any $I\in\T.$ \cma{Let 
$$
L=\sup\{\avm{\varphi}R\colon R,\,R\in\T,\,R\supset I\}.
$$ 
}
Then
$$
L=\inf_IN_\T\varphi.
$$
\end{lemma}
\begin{proof}
The inequality $L\le \inf_IN_\T\varphi$ is obvious. To show the converse, assume that $L<\inf_IN_\T\varphi.$ Then $\mu$-almost every point of $I$ lies in some maximal subset $J\subset I$ such that $J\in\T$ and $\avm{\varphi}J>L.$ Let $\{J_k\}$ be the collection of these maximal tree elements; they \cma{cover $I$ and} are disjoint up to measure zero. Therefore,
$$
\avm{\varphi}I=\frac1{\mu(I)}\,\sum_k \mu(J_k) \avm{\varphi}{J_k}> \frac1{\mu(I)}\,\sum_k \mu(J_k) L=L,
$$
which is a contradiction since $\avm{\varphi}I\le L$ by the definition of $L.$
\end{proof}
\begin{remark}
Since it was first done in~\cite{nt}, fixing the ``external maximal function'' $L$ when defining Bellman functions for the dyadic maximal operator has become canonical. The result of Lemma~\ref{key_obs} makes this approach particularly advantageous for settings where the infimum of the maximal function is involved, such as $\BLO$ and $A_1.$ Note that this result also holds with the usual maximal operator $M$ in place of $N.$ However, this equality is false in general for the usual (non-dyadic) Hardy--Littlewood maximal operator.
\end{remark}

In light of Lemma~\ref{key_obs}, to show that $N$ maps $\BMO^d$ into $\BLO^d,$ it is necessary and sufficient to show that
$$
\bel{B}^n(x,L)\le L+c
$$
for some finite $c.$ Furthermore, the best function $\Phi$ from Theorem~\ref{main_th} is given by
\eq[est1]{
\Phi_n(t)=\sup_{(x,L)\in S;~L-x_1=t}\big(\bel{B}^n(x,L)-L\big).
}

\subsection{Properties of the Bellman function}
\label{properties}
Let us make three basic observations about the function $\bel{B}^n.$ First, we \cma{have} an {\it a priori} boundary condition for $\bel{B}^n.$
\begin{lemma}
\label{bc}
For all $x_1\le L,$
$$
\bel{B}^n(x_1,x_1^2,L)=L.
$$
\end{lemma}
\begin{proof}
Every element $\varphi$ of $E_{(x_1,x_1^2),Q,L}$ \cma{almost everywhere on $Q$ takes the constant value $x_1\le L,$} thus, $N\varphi|_{Q}=\inf_{Q}N\varphi=L.$
\end{proof}

Second, we show that $\bel{B}^n$ possesses a \cma{special} restricted concavity on $S.$
\begin{lemma}
\label{mi}
Take $x^-,x^+\in\Omega.$ Let $x=(1-2^{-n})x^-+2^{-n}x^+$ and assume that $x\in\Omega.$ Take any $L\ge x_1$ and let $L^\pm=\max\{x_1^\pm,L\}.$ Then
\eq[mi1]{
\bel{B}^n(x,L)\ge (1-2^{-n})\bel{B}^n(x^-,L^-)+2^{-n}\bel{B}^n(x^+,L^+).
}
\end{lemma}
\begin{proof}
Fix $Q\in\mathcal{D}.$ Let $\{\varphi_j^\pm\}$ be a sequence of functions from $E_{x^\pm,L^\pm,Q}$ such that
$$
\av{N\varphi_j^\pm}Q\to \bel{B}^n(x^\pm,L^\pm).
$$
Let $\{Q_k\}_{k=1}^{2^n}$ be the dyadic subcubes of $Q$ of the first generation. Define a new sequence $\{\varphi_j\}$ on $\rn$  as follows: for each $j,$ on $Q_k,$ $1\le k\le 2^n-1,$ let $\varphi_j$ be $\varphi^-_j|_{Q}$ rescaled to $Q_k;$ on $Q_{2^n},$ let $\varphi_j$ be $\varphi^+_j|_Q$ rescaled to $Q_{2^n};$ and on $\rn\setminus Q,$ let $\varphi_j=L.$ By construction, each $\varphi_j\in E_{x,L,Q}.$
Furthermore,
$$
\bel{B}^n(x,L)\ge\av{N\varphi_j}Q=(1-2^{-n})\av{N\varphi^-_j}Q+2^{-n}\av{N\varphi^+_j}Q.
$$
Since the right-hand side converges to $(1-2^{-n})\bel{B}^n(x^-,L^-)+2^{-n}\bel{B}^n(x^+,L^+)$ as $j\to\infty,$ the proof is complete.
\end{proof}
Our third observation concerns the additive homogeneity of the function $\bel{B}^n.$ Consider the following translation operator on $\mathbb{R}^2$: for $a\in\mathbb{R},$ let
\eq[shift]{
T_a(x_1,x_2)=(x_1-a,x_2-2ax_1+a^2)\cma{.}
}
Clearly, for any $a,$ $T_a$ is a bijection on $\Omega.$
\begin{lemma}
\label{homogeneity}
$$
\bel{B}^n(x,L)=L+\bel{B}^n(T_{L}x,0).
$$
\end{lemma}
\begin{proof}
If $\vf\in E_{x,L,Q},$ then $\tilde\vf:=\vf-a\in E_{\tilde x,\tilde L,Q},$ where $(\tilde x,\tilde L)=(T_a x, L-a).$ Since $N\tilde\vf=N\vf-a$, we get 
$\bel{B}^\cma{n}(\tilde x,\tilde L)=\bel{B}^\cma{n}(x,L)-a.$ Now, take $a=L$. 
\end{proof}

According to Lemma~\ref{homogeneity}, we can rewrite \eqref{est1} as follows:
\eq[est2]{
\Phi_n(t)=\sup_{x\in\Omega;~x_1=-t}\bel{B}^n(x,0).
}

\subsection{$\cma{\boldsymbol{\alpha}}$-concavity and Bellman induction} 
We will need the following definition from~\cite{alpha_trees}:
\begin{definition}
\label{def}
If $\alpha\in\big(0,\frac12\big],$ a function $F$ on $\Omega$ is called $\alpha$-concave if
\eq[601]{
F((1-\beta) x^-+\beta x^+)\ge (1-\beta)F(x^-)+\beta F(x^+),\\
}
for any $\beta\in\big[\alpha,\frac12\big]$ and any two points $x^\pm\in\Omega$ such that $(1-\beta) x^-+\beta x^+\in\Omega.$
\end{definition}

We will also need a simple lemma whose elementary proof can be found in~\cite{alpha_trees}. (Specifically, this is the first step in the proof of Lemma~2.4 of that paper.)
\begin{lemma}
\label{one_step}
Take $\alpha\in(0,1/2]$ and let $\T$ be an $\alpha$-tree on a measure space $(X,\mu).$ Let $\varphi\in\BMO(\T)$. If $F$ is an $\alpha$-concave function on $\Omega,$ then for any $I\in\T,$
\eq[ind]{
F\big(\avm{\varphi}I,\avm{\varphi^2}I\big)\ge\frac1{\cma{\mu(I)}}\sum_{J\in\T_1(I)}\cma{\mu(J)}\,F\big(\avm{\varphi}J,\avm{\varphi^2}J\big).
}
\end{lemma}
Our next lemma shows how $\alpha$-concave functions can be used to bound the functional $\avm{N_\T\varphi}K$ for $K\in\T$ in terms of $\avm{\varphi}K,$ $ \avm{\varphi^2}K,$ $\inf_KN_\T\varphi,$ and (implicitly) $\|\varphi\|_{\BMO(\T(K))}.$ The process it implements is commonly referred to as Bellman induction.
\begin{lemma}
\label{induction}
Fix $\alpha\in\cma{(}0,1/2],$ and let $\T$ be an $\alpha$-tree on a measure space $(X,\mu).$  Let $\{A(\,\cdot\,;L)\}_{L\in\mathbb{R}},$ be a family of functions on $\Omega,$ such that for each $L,$ 
\ben[label=(\arabic*),font=\upshape]
\item
\label{cond1}
$A(\,\cdot\,;L)$ is $\alpha$-concave on $\Omega.$ 
\item
\label{cond2}
For all \cma{$x_1\ge L,$} $A(x; L)=A(x; x_1).$
\item
\label{cond3}
For all $x_1\le L,$ $A(x; L)\ge L.$
\een
Take $K\in\T$ and any function $\varphi\in L^1_{loc}(X)$ such that $N_\T\varphi$ is not identically infinite on $K,$ $\varphi|_K\in\BMO\big(\T(K)\big)$ and $\|\varphi\|_{\BMO(\T(K))}\le 1.$ 
Then 
\eq[direct]{
\avm{N_\T\varphi}{K} \le A(\avm{\varphi}K,\avm{\varphi^2}K;\,\inf_KN_\T\varphi).
}

\end{lemma}

\begin{proof}
For all $J\in\T(K),$ let us write 
$$
P_J(\varphi)=(\avm{\varphi}J,\avm{\varphi^2}J\big),\qquad L_J(\varphi)=\sup_{J\subset R\in\cma{\T}}\avm{\varphi}R.
$$  
Note that if $I\in\T(K)$ and $J\in\T_1(I),$ then $L_J(\varphi)=\max\{L_I(\varphi), \avm{\varphi}J\}.$

Fix an $L\in\mathbb{R}.$ 
Using Lemma~\ref{one_step} in conjunction with property~\ref{C1} of $A$, and then property~\ref{C2} of $A,$ repeating this process $m$ times, and finally applying property~\ref{C3}, we obtain
\begin{align*}
A(P_K(\varphi); L_K(\varphi))
&\ge\frac1{\cma{\mu(K)}}\sum_{J\in\T_1(K)}\cma{\mu(J)}\,A(P_J(\varphi); L_K(\varphi))
=\frac1{\cma{\mu(K)}}\sum_{J\in\T_1(K)}\cma{\mu(J)}\,A(P_J(\varphi); L_J(\varphi))\\
&\ge \frac1{\cma{\mu(K)}}\sum_{J\in\T_1(K)}\sum_{\cma{I}\in\T_1(J)}\cma{\mu(I)}\,A(P_\cma{I}(\varphi); L_\cma{I}(\varphi))
=\frac1{\cma{\mu(K)}}\sum_{J\in\T_2(K)}\cma{\mu(J)}\,A(P_J(\varphi); L_J(\varphi))\\
&...\\
&\ge \frac1{\cma{\mu(K)}}\sum_{J\in\T_m(K)}\cma{\mu(J)}\,A(P_J(\varphi); L_J(\varphi))\ge\frac1{\cma{\mu(K)}}\sum_{J\in\T_m(K)} \cma{\mu(J)}\,L_J(\varphi).
\end{align*}
Now, \cma{for $m\ge0,$ let $\varphi_m$ be the conditional expectation with respect to the $\sigma$-algebra generated by $\T_m(K),$ i.e.,}
$$
\varphi_m=\mathbb{E}(\varphi|\T_m(K))=\sum_{J\in\T_m(K)}\avm{\varphi}J\,\chi^{}_J.
$$
Then, for every $J\in\T_m(K),$ $L_J(\varphi)$ is the constant value of $N_{\T}(\varphi_m)$ on $J.$ Therefore, we have
$$
A(P_K(\varphi); L_K(\varphi))\ge \avm{N_\T(\varphi_m)}K\cma{.}
$$
By Lemma~\ref{key_obs}, $L_K(\varphi)=\inf_KN_{\T}\varphi.$
Since $N_{\T}(\varphi_m)$ is increasing $\mu$-a.e. to $N_{\T}\varphi,$ inequality~\eqref{direct} follows by the monotone convergence theorem. 
\end{proof}

\subsection{The Bellman candidate}
\label{bell_c}
We now present a family $\{A(\,\cdot\,; L)\}_{L\in\mathbb{R}}$ satisfying the conditions of Lemma~\ref{induction}. As we will see shortly, it will suffice to specify only the member of this family corresponding to $L=0.$ To give our definition, we need to split $\Omega$ into a union of special subdomains.

Let 
\eq[tau]{
\tau=\frac1{\sqrt\alpha}-\sqrt\alpha,\qquad \cma{p_0=\frac12\sqrt\alpha+\frac1{2\sqrt\alpha}-1};\qquad \cma{p_k=p_0-k\tau},~\cma{k\ge1}.
}
\cma{Using the parabolic shift $T_a$ defined by~\eqref{shift}, we can write} $(p_k,p_k^2+1)=T_{(k-1)\tau}(p_1,p_1^2+1).$

Let
\eq[domains]{
\begin{aligned}
&\Omega_+=\{x\in\Omega\colon x_1\ge0\},\\
&\Omega_0=\{x\in\Omega\colon x_1\le0, x_2\le 1\},\\
&\cma{\Omega_1=\{x\in\Omega\colon 1\le x_2\le (2p_1+\tau)x_1-p_1^2-\tau p_1+1\}},\\
&\cma{\Omega_2=\{x\in\Omega\colon (2p_1+\tau)x_1-p_1^2-\tau p_1+1\le x_2\le-2\tau x_1-\tau^2+1\}},\\
&\Omega_{2k+1}=T_{k\tau}\Omega_1,~k\ge1,\\
&\Omega_{2k+2}=T_{k\tau}\Omega_2,~k\ge1.
\end{aligned}
}
Figure~\ref{regions_pic} shows the first several subdomains for $\alpha=\frac14.$
\begin{figure}[h]
\centering{
\includegraphics[width=16cm]{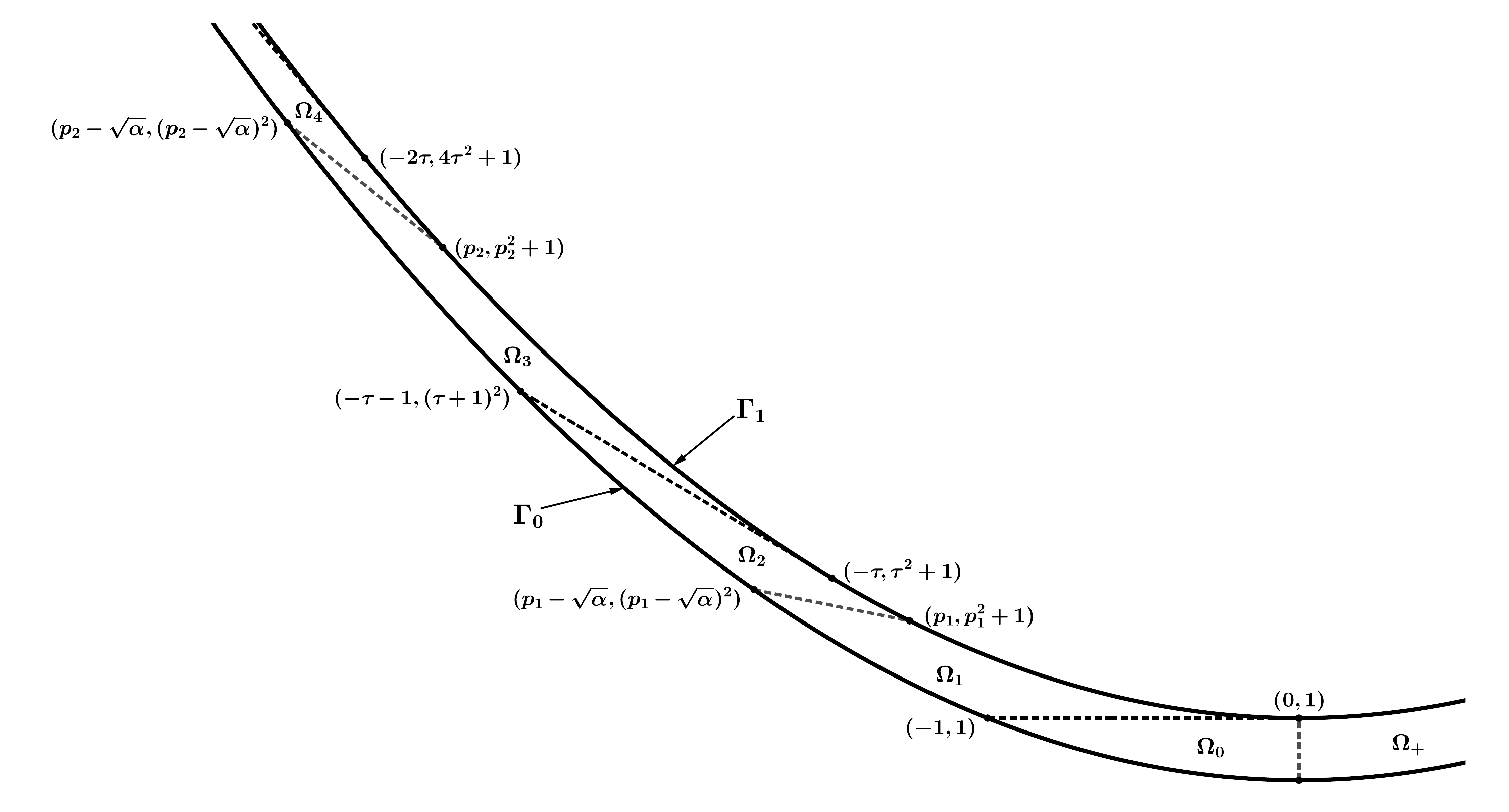}
\caption{The splitting $\Omega=...\cup\Omega_4\cup\Omega_3\cup\Omega_2\cup\Omega_1\cup\Omega_0\cup\Omega_+$ for $\alpha=\frac14$}
\label{regions_pic}
}
\end{figure}

We will also find it useful to write
$$
\Omega_-=\bigcup_{k=1}^\infty\Omega_k.
$$
In this notation,
$
\Omega=\Omega_-\cup\Omega_0\cup\Omega_+.
$

Next, we define a function $B$ on $\Omega$ that will be used to construct the family $\{A(\,\cdot\,; L)\}.$

\noindent In $\Omega_+,$ let 
\eq[d_+]{
B(x)=x_1+\sqrt{x_2-x_1^2}.
}
In $\Omega_0,$ let 
\eq[d_0]{
B(x)=x_1+\sqrt{x_2}.
}
To define $B$ in $\Omega_1,$ for each $s\in[\sqrt\alpha,1]$ let 
\eq[vu_1]{
v(s)=\frac12\Big[3s-\frac1s\Big]-1, \qquad u(s)=v(s)-s.
}
Consider the family of line segments $\{\ell_s\},$ where each $\ell_s$ connects the points $(u(s),u^2(s))$ and $(v(s),v^2(s)+1).$ It is easy to check that these segments foliate $\Omega_1,$ meaning each segment is contained in $\Omega_1$ and for every point $x\in\Omega_1$ there exists a unique number $s=s(x)\in[\sqrt\alpha,1]$ such that $x\in\ell_s.$
Now, let
\eq[d_1]{
B(x)=\frac s2\,(1+s^2)\,\frac{x_1-u}{v-u}=\frac12\,(1+s^2)(x_1-u).
}
Observe that $B$ is linear along the segment $\ell_s.$

To define $B$ in $\Omega_2,$ for each $s\in[\alpha,\sqrt\alpha]$ let 
\eq[vu_2]{
v(s)=\frac12\Big[\frac s\alpha+\frac\alpha s\Big]-\tau-1,\qquad  u(s)=v(s)-\frac\alpha s.
}
Again, consider the family of line segments $\{\ell_s\}$ connecting the points $(u(s),u^2(s))$ and $(v(s),v^2(s)+1).$ As before, it is easy to check that these segments foliate $\Omega_2.$ Let $s=s(x)$ be  the unique number in $[\alpha,\sqrt\alpha]$ such that $x\in\ell_s.$ Now, let
\cma{
\eq[d_2]{
B(x)=\frac s{2}\,\Big(1+\frac{\alpha^2}{s^2}\Big)\,\frac{x_1-u}{v-u}=\frac\alpha{2}\,\Big(1+\frac{s^2}{\alpha^2}\Big)\,(x_1-u).
}
}
Again, $B$ is linear along each $\ell_s.$

For $x\in\Omega_{2k+1}\cup\Omega_{2k+2}, k\ge1,$ \cma{we define $B$ using our parabolic shift:
\eq[d_3+]{
B(x)=\alpha^k\,B(T_{-k\tau}x).
}
The auxiliary functions are defined accordingly:}
\eq[s_3+]{
s(x)=\alpha^ks(T_{-k\tau}x),\quad v(s)=v\Big(\frac s{\alpha^k}\Big)-k\tau,\quad u(s)=u\Big(\frac s{\alpha^k}\Big)-k\tau.
}
Thus, we have $s(x)\in[\alpha^{k+\frac12},\alpha^{k}]$ if $x\in\Omega_{2k+1}$ and $s(x)\in[\alpha^{k+1},\alpha^{k+\frac12}]$ if $x\in\Omega_{2k+2}.$ Furthermore, with these definitions, every $x\in\Omega_{2k+1}\cup\Omega_{2k+2}$ lies on a unique segment $\ell_{s}$ connecting the points $(u,u^2)$ and $(v,v^2+1)$ and such that $B$ is linear along $\ell_s.$ \cma{These segments are mapped into each other by the parabolic shift:}
$$
\ell_s=T_{k\tau}\big(\ell_{s\alpha^{-k}}\big).
$$

Let us collect a few observations about the function $B.$ 
\begin{lemma}
\label{b_bound}
Let $b(p)=B(p,p^2+1).$
Then
\eq[bp]{
b(p)=
\begin{cases}
p+1,& p\ge 0;\\
\alpha^{k}\,f(p+k\tau+1),& p_{k+1}\le p\le -k\tau,~k\ge0;\\
\alpha^k(p+k\tau +1), & -k\tau\le p\le p_k,~k\ge1,
\end{cases}
}
where the function $f$ is given by
\eq[gy]{
f(y)=\frac1{27}\,\big(2y^3+2y^2\sqrt{y^2+3}+9y+6\sqrt{y^2+3}\big).
}
\end{lemma}
\begin{proof}
Let $P=(p,p^2+1).$ We have $P\in\Omega_-\cup\Omega_+.$ If $P\in\Omega_+,$ then \cma{by definition} $b(p)=p+1,$ as claimed on the first line of~\eqref{bp}.

If $P\in\Omega_1,$ \eqref{vu_1} and~\eqref{d_1} give
$
\cma{p=\frac12\,\big(3s-\frac1{s}\big)-1,}$ $b(p)=\frac12\,(s+s^3),
$
where $s=s(\cma{P}).$ It is easy to write $b$ explicitly in terms of $p.$ Specifically, we obtain $b(p)=f(p+1)$ with $f$ given by~\eqref{gy}. Taking into account~\eqref{d_3+}, we get the second line of~\eqref{bp}.

If $P\in\Omega_2,$ then $b$ is an affine function of $p.$ Therefore, the same is true in every $\Omega_{2k}.$ After a small bit of algebra, we see that for $k\ge1$ and $P\in\Omega_{2k},$
$
b(p)=\alpha^k(p+k\tau+1),
$
completing the proof.
\end{proof}

\begin{lemma}
\label{b_inc}
The function $B$ is increasing in $x_2.$ 
\end{lemma}
\begin{proof}
The statement is obvious in $\Omega_+\cup\Omega_0.$ If $x\in\Omega_1,$ then direct differentiation of $B$ from~\eqref{d_1} with respect to $s$ gives
$$
2\frac{\partial B}{\partial s}=-u_s(1+s^2)+2s(x_1-u).
$$
We have $0\le x_1-u\le s.$
In addition, from~\eqref{vu_1} we have $\cma{u_s=\frac12\big(1+\frac1{s^2}\big).}$ Therefore,
\begin{align*}
2\frac{\partial B}{\partial s}&\le -\frac12\,\Big(1+\frac1{s^2}\Big)(1+s^2)+2s^2=\frac1{2s^2}\,(4s^4-(1+s^2)^2)=\frac1{2s^2}\,(s^2-1)(\cma{3}s^2+1)\le0,
\end{align*}
since $s\le 1.$ It is clear from geometry of the foliation that $s$ is decreasing in $x_2,$ hence we can conclude that $B$ is increasing in $x_2.$ 

If $x\in\Omega_2,$ the argument is identical, except we have to replace $s$ with $\frac\alpha s$ throughout. Finally, the statement for the rest of $\Omega$ follows from formula~\eqref{d_3+}, since only the second component of $T_{-k\tau}x$ depends on $x_2,$ and that component is increasing in $x_2.$
\end{proof}
The main fact about $B$ is contained in the following lemma. Its (somewhat technical) proof is given in Section~\ref{alpha_concavity}.
\begin{lemma}
\label{a-conc}
The function $B$ defined by~\eqref{domains}--\eqref{d_3+} is $\alpha$-concave on $\Omega.$
\end{lemma}
We now define the family $\{A(\,\cdot\,; L)\}$ and verify its properties. For $L\in\mathbb{R},$ let
\eq[def_a]{
A(x; L)=L+B(T_L x)=L+B(x_1-L,x_2-2x_1L+L^2).
}
\begin{lemma}
\label{a_ver}
The family~\eqref{def_a} has properties \textup{(1)--(3)} from Lemma~\textup{\ref{induction}}.
\end{lemma}
\begin{proof}~
\ben[leftmargin=*]
\item
We verify Definition~\ref{def} for each $A(\,\cdot\,; L), L\in\mathbb{R},$ using~\eqref{def_a}, the linearity of the shift $T_a,$ and Lemma~\ref{a-conc}: 
\begin{align*}
A((1-\beta) x^-+\beta x^+; L)&=L+B\big(T_L((1-\beta) x^-+\beta x^+)\big)=L+B((1-\beta) T_Lx^-+\beta T_L x^+)\\
&\ge (1-\beta) (L+B(T_L x^-))+ \beta (L+B(T_L x^+))\\
&=(1-\beta) A(x^-; L)+ \beta A(x^+; L).
\end{align*}
\item
If $x_1\cma{\ge}L,$ then $T_Lx\in\Omega_+,$ so
\begin{align*}
A(x; L)&=L+\big(x_1-L+\sqrt{(x_2-2x_1L+L^2)-(x_1-L)^2}\big)\\
&=x_1+\sqrt{x_2-x_1^2}=x_1+B(0,x_2-x_1^2)=A(x;x_1).
\end{align*}
\item
Since $B\ge0$ on $\Omega,$ which is evident from \eqref{d_+}--\eqref{d_3+}, we have $A(x; L)\ge L.$ \qedhere
\een
\end{proof}

We are now in a position to prove the inequalities in Theorem~\ref{main_th_trees} and, thus, those in Theorem~\ref{main_th}.

\begin{proof}[Proof of Theorem~\textup{\ref{main_th_trees}}]
\cma{Letting in Lemma~\ref{induction}} $x_1=\avm{\varphi}J,$  $x_2=\avm{\cma{\varphi^2}}J,$ $L=\inf_J N_\T\varphi,$ and $t=L-x_1,$ we have
$$
\avm{N_\T\varphi}J \le A(x; L)=L+B(T_L x)\le L+B(-t,t^2+1)=L+b(-t),
$$
where the second inequality is due to Lemma~\ref{b_inc} \cma{and the function $b$ is given by~\eqref{bp} and~\eqref{gy}.} Therefore, we can take
$$
\mathcal{F}_\alpha(t)=b(-t).
$$
Note that $\mathcal{F}_\alpha(k\tau)=b(-k\tau)=\alpha^k.$ \cma{The elementary verification of the fact that $b$ is increasing and convex on $[0,\infty)$ -- and, thus, that $\mathcal{F}_\alpha$ is decreasing and convex  -- is left as an exercise.}
\end{proof}
\subsection{Main Bellman theorem}
As mentioned in the introduction, the inequalities of Theorem~\ref{main_th} (but not their sharpness) now follow as an immediate corollary~\cma{of Theorem~\ref{main_th_trees}}. However, we can now state a stronger result that fully captures the connection between the function $B$ -- and, thus, the family $\{A(\,\cdot\,; L)\}_{L\in\mathbb{R}}$ -- and the Bellman function $\bel{B}^n$ defined in~\eqref{b_test}, \eqref{b_main}. This result can be viewed as the main finding of the paper.

\begin{theorem}
\label{main_bellman_th}
For each $(x,L)\in S,$ let $A^n(x;L)$ be defined by~\eqref{domains}--\eqref{d_3+} and~\eqref{def_a} for $\alpha=2^{-n}.$ Then
$$
\bel{B}^n(x, L)= A^n(x; L)\quad\text{for all}~(x,L)\in S.
$$
\end{theorem}
As is customary, the proof of this theorem consists of two lemmas.
\begin{lemma}
\label{B<A}
$$
\bel{B}^n(x, L)\le A^n(x; L)\quad\text{for all}~(x,L)\in S.
$$
\end{lemma}
\begin{proof}
Take any $(x,L)\in S$ and any dyadic cube $Q.$ As mentioned before, the set $E_{x,L,Q}$ is non-empty. Take  $\varphi\in E_{x,L,Q}.$ Let $R$ be the dyadic parent of $Q.$ Define a new function $\tilde\varphi$ by setting $\tilde\varphi=\varphi\,\chi_Q+a\,\chi_{R\setminus Q},$ where $a:=\frac{L|R|-x_1|Q|}{|R\setminus Q|}=\frac{2^nL-x_1}{2^n-1}.$ Let $\T$ be the collection of all dyadic subcubes of $R.$ Then $\T$ is a $2^{-n}$-tree on $R$ (equipped with the Lebesgue measure), and $N\varphi|_Q=N_{\T}\tilde\varphi|_Q.$ By Lemma~\ref{induction} with $K=Q,$ 
$$
\av{N\varphi}Q=\av{N_\T\tilde\varphi}Q\le A^n(\av{\tilde\varphi}Q,\av{\tilde\varphi^2}Q;\,\inf_QN_\T\tilde\varphi)=A^n(x;L).
$$
Now, take the supremum in the left-hand side over all $\varphi \in E_{x,L,Q}.$
\end{proof}

\begin{lemma}
\label{B>A}
$$
\bel{B}^n(x, L)\ge A^n(x; L)\quad\text{for all}~(x,L)\in S.
$$
\end{lemma}
To prove this lemma, and thus Theorem~\ref{main_bellman_th}, we will use the concavity properties of the Bellman function $\bel{B}^n$ \cma{and} the geometric structure of the candidate $A^n.$ This is done in Section~\ref{optimizers}.
\subsection{Bellman majorants}
We now present another family of functions on $\Omega$ that verifies the conditions of Lemma~\ref{induction} and thus bounds the Bellman function $\bel{B}^n$ from above. Let
$$
B_0(x)=
\begin{cases}
x_1+\sqrt x_2,&x\in\Omega_-\cup\Omega_0,\\
x_1+\sqrt{x_2-x_1^2},& x\in\Omega_+
\end{cases}
$$
and
$$
A_0(x; L)=L+B_0(T_L x).
$$
Observe that each $A_0(\,\cdot\,; L)$ is concave in the convex region $\{x: x_2\ge x_1^2\}$ which implies that it is $\alpha$-concave on $\Omega.$ Its maximum on $\Omega\cap\{x_1\le L\}$ is attained at the point $(L,L^2+1).$ Therefore, 
\eq[b_0]{
\bel{B}^n(x,L)\le A_0(x; L)\le A_0(L,L^2+1;L)=L+1.
}
Thus, the family $\{A_0(\,\cdot\,; L)\}_L$ gives the sharp norm constant 1 for the dyadic maximal function. It also has the obvious advantage of being simple and explicit. In fact, it coincides with $\bel{B}^n$ in the region $T_{-L}\Omega_0.$ However, it produces only crude bounds for $\bel{B}^n$ away from that region. In particular, the function $\Phi_n(t)$ that this family would yield in Theorem~\ref{main_th} is
$$
\Phi_n(t)=\sqrt{t^2+1}-t,
$$
which is far from the sharp exponential dimensional decay of the true function $\Phi_n.$ 

At the cost of added complexity, one can produce a better majorant family $\{A_k(\,\cdot\,;L)\}_L$ by taking any $k\ge1,$ cutting the ``true'' Bellman candidate $B$ defined by~\eqref{d_1}--\eqref{d_3+} off after $\Omega_k,$ and \cma{extending this cut-off beyond $\Omega_k$ by the same analytical expression as in $\Omega_k.$ This will give an $\alpha$-concave function and, upon setting setting 
$$A_k(x; L)=L+B_k(T_Lx),$$ 
a new family of majorants. None of these majorants would yield the decay of the function~$\Phi_n,$ though they would converge pointwise to the true candidate $B$ as $k\to\infty.$}

\section{The converse inequality and optimizers}
\label{optimizers}

In this section, we prove Lemma~\ref{B>A} and thus complete the proof of Theorems~\ref{main_bellman_th} and~\ref{main_th}. As shown below, it will be enough to establish the special case $L=0.$ Thus, all consideration will be restricted to the following domain:
$$
\Omega_*:=\Omega_-\cup\Omega_0=\bigcup_{k=0}^\infty\Omega_k.
$$
To streamline notation, throughout this section let us write $Q_0=(0,1)^n,$ $\mathcal{B}(x)=\bel{B}^n(x,0),$ and $F_x=E_{x,0,Q_0}$ for $x\in\Omega_*.$ We will also reuse the earlier notation $\tau=\frac{1-\alpha}{\sqrt\alpha}$ with $\alpha=2^{-n}.$ Thus, in this section, $\tau=2^{n/2}-2^{-n/2}.$
\begin{lemma}
\label{l=0}
\eq[BB]{
\mathcal{B}(x)\ge B(x)\quad\text{for all}~x\in\Omega_*.
}
\end{lemma}
Inequality~\eqref{BB} immediately implies Lemma~\ref{B>A}. Indeed, by Lemma~\ref{homogeneity} and formula~\eqref{def_a},
$$
\bel{B}^n(x,L)=L+\mathcal{B}(T_Lx)\ge L+B(T_Lx)=A^n(x;L),
$$
as claimed.

\begin{remark}
\label{b=0}
Note that on $\Gamma_0\cap\Omega_*$ both $\cma{\mathcal{B}}$ and $B$ are 0. For $\mathcal{B},$ this is the result of Lemma~\ref{bc}. For $B,$ this follows from~\eqref{d_0} in $\Omega_0;$ \eqref{d_1} in $\Omega_1;$ \eqref{d_2} in $\Omega_2;$ and \eqref{d_3+} in the rest of $\Omega_*.$
\end{remark}

\begin{definition}
\label{d_opt}
If $x\in\Omega_*,$ we say that a sequence of functions $\{\varphi_j\}$ on $\rn$ is an optimizing sequence for $B$ at $x$ if each $\varphi_j\in F_x$ and
\eq[lim]{
\av{N\varphi_j}{Q_0}\to B(x)~\text{as}~j\to\infty.
}
\end{definition}
The standard way to show a statement like~\eqref{BB} is to demonstrate an optimizing sequence for every point $x\in\Omega_*.$ In many dyadic problems, the Bellman function is seen to be concave in a certain sense directly from definition, which might allow one to get away with finding optimizers only on the boundary of the domain (cf.~\cite{alpha_trees}). In our current setting, we are able to go further still: combining the concavity properties of the function $\bel{B}^n$ given in Lemma~\ref{mi} with the geometric structure of the function $B,$ we can prove~Lemma~\ref{l=0} without relying on explicit optimizers at all. However, since such optimizers are of independent interest, we do present an optimizing sequence for the key point $x=(0,1)$ later in the section. (Since this is the sequence on which the $\BMO\to\BLO$ norm of the operator $N$ is attained in the limit, we call it the norm-optimizing sequence.)

We first need to establish~\eqref{BB} for the case when $x$ is on the boundary of $\Omega_*.$ Remark~\ref{b=0} leaves us with the right boundary $\{(0,y)\colon 0\le y\le 1\}$ and the top boundary $\Gamma_1\cap\Omega_*$ to consider. The arguments for these two cases are somewhat different, but they both really heavily on the result of Lemma~\ref{mi} and the related fact that $\mathcal{B}$ is locally concave on $\Omega_*.$ \cma{(We call a function locally concave on a domain, if it is concave on any convex subdomain.)}

\begin{lemma}
\label{loc_conc}
$\mathcal{B}$ is locally concave on $\Omega_*.$
\end{lemma}
\begin{proof}
Take any two points $x^-, x^+\in\Omega_*$ such that the entire line segment $[x^-,x^+]$ lies in $\Omega_*.$ Let $\{\varphi^-_j\}$ and $\{\varphi^+_j\}$ be the optimizing sequences for $B$ at $x^-$ and $x^+,$ respectively. Take any $\gamma\in(0,1)$ and split $Q_0$ into a union of two sets, one of measure $1-\gamma$ and the other of measure $\gamma.$ Each of the two sets can be written (up to measure zero) as a union of disjoint dyadic subcubes of $Q_0.$ Thus, 
$$
Q_0\approx \Big(\bigcup_k Q^-_k\Big)\cup\Big(\bigcup_k Q^+_k\Big),
$$
where $\sum_k|Q^-_k|=1-\gamma,$ $\sum_k|Q^+_k|=\gamma,$ and ``$\approx$'' means equality up to measure zero.

Define a new sequence $\{\varphi_j\}$ on $\rn$ 
by setting $\varphi_j$ to be 0 on $\rn\setminus Q_0;$ \cma{a rescaled copy of $\varphi_j^-$ on each $Q_k^-;$ and a rescaled copy of $\varphi_j^+$ on each $Q_k^+.$} Clearly, each $\varphi_j\in F_{(1-\gamma)x^-+\gamma x^+}$ and 
\begin{align*}
\mathcal{B}((1-\gamma)x^-+\gamma x^+)&\ge\av{N\varphi_j}{Q_0}=\sum_k|Q_k^-|\av{N\varphi_j}{Q_k^-}+\sum_k|Q_k^+|\av{N\varphi_j}{Q_k^+}\\
&=(1-\gamma)\av{N\varphi^-_j}{Q_0}+\gamma \av{N\varphi^+_j}{Q_0}.
\end{align*}
The right-hand side converges to $(1-\gamma)\mathcal{B}(x^-)+\gamma\mathcal{B}(x^+)$ as $j\to\infty,$ proving the claim.
\end{proof}

\begin{lemma}
\label{x_1=0}
$$
\mathcal{B}(0, y)\ge B(0, y)\quad\text{for all}~0\le y\le 1.
$$
\end{lemma}
\begin{proof}
It suffices to show the claim for positive $y$, since for $y=0$ both $\mathcal{B}(0,y)$ and $B(0,y)$ vanish. Take a small parameter $\delta>0$. Let $x^-=(2^{-n}\delta,y)$ and $x^+=(-(1-2^{-n})\delta,y).$ Using Lemma~\ref{mi} with $L=L^+=0$ and $L^-=2^{-n}\delta,$ we get
\eq[concavity_1]{
 \mathcal{B}(0,y)\geq (1-2^{-n})\bel{B}^n(2^{-n}\delta,y,2^{-n}\delta)+2^{-n}\mathcal{B}(-(1-2^{-n})\delta,y).
}
By Lemma~\ref{homogeneity}, we have
$$ \bel{B}^n(2^{-n}\delta,y,2^{-n}\delta)=2^{-n}\delta+\bel{B}^n(0,y-2^{-2n}\delta^2,0)=2^{-n}\delta+\mathcal{B}(0,y-2^{-2n}\delta^2),$$
while by Lemma~\ref{loc_conc}, we have
\begin{align*}
 \mathcal{B}(-(1-2^{-n})\delta,y)&\geq \frac{\sqrt{y}-(1-2^{-n})\delta}{\sqrt{y}}\mathcal{B}(0,y)+\frac{(1-2^{-n})\delta}{\sqrt{y}}\mathcal{B}(-\sqrt{y},y)\\
 &=\frac{\sqrt{y}-(1-2^{-n})\delta}{\sqrt{y}}\mathcal{B}(0,y),
 \end{align*}
since $\mathcal{B}$ vanishes on the lower boundary of $\Omega_*.$

\cma{Plugging these two relations into \eqref{concavity_1} and dividing by $1-2^{-n}$} yields \begin{equation}\label{concavity_2}
 \mathcal{B}(0,y)-\mathcal{B}(0,y-2^{-2n}\delta^2)\geq -\frac{2^{-n}\delta}{\sqrt{y}}\mathcal{B}(0,y)+2^{-n}\delta.
\end{equation}
Now, recall that the function $y\mapsto \mathcal{B}(0,y)$ is concave on $[0,1]$; in particular, its one-sided derivatives exist and are finite on $(0,1)$. Assuming $y<1$, if we divide both sides of \eqref{concavity_2} by $\delta$ and let $\delta\to 0$, the left-hand side will vanish and we obtain $\mathcal{B}(0,y)\geq \sqrt{y}=B(0,y),$ as desired. Finally, for $y=1,$ we apply \eqref{concavity_2} together with the estimate $\mathcal{B}(0,1-2^{-2n}\delta^2)\geq \sqrt{1-2^{-2n}\delta^2}$ we have just established. Letting $\delta\to 0$ gives $\mathcal{B}(0,1)\geq 1,$ and the claim is proved.
\end{proof}

\begin{lemma}
\label{gamma_1}
$$
\mathcal{B}(v, v^2+1)\ge B(v,v^2+1)\quad\text{for all}~v\le0.
$$
\end{lemma}
\begin{proof}
We first show the assertion for $P\cma{=}(v,v^2+1)\in \Omega_1\cup \Omega_2$. \cma{If} $P\in\Omega_1,$ then, according to~\eqref{vu_1}, $v=\frac12(3s-\frac1s)-1$ for some $s\in[2^{-n/2},1].$ Let $u=v-s,$ $v^+=v-s+\frac1s,$ and consider the auxiliary points
$$ Q=(v^+,(v^+)^2+1),\qquad R=(u,u^2).$$
It is easy to check that $P$ belongs to the segment $\cma{[Q,R]}$ (in fact, we have $ P=s^2Q+(1-s^2)R$) and that the entire segment $\cma{[P,R]}$ is contained in $\Omega$. 
Since $s^2\ge 2^{-n},$ there is a point $R_1$ belonging to $\cma{[P,R]}$ such that $P=(1-2^{-n})R_1+2^{-n}Q$. Note that $\cma{v^+}=\frac12(s+\frac1s)-1\ge0.$ Hence, Lemma~\ref{mi} with $L=L^-=0$ and $L^+=\cma{v^+}$ yields
$$ 
\mathcal{B}(P)\geq (1-2^{-n})\mathcal{B}(R_1)+2^{-n}\bel{B}^n(Q, \cma{v^+}).
$$
By Lemma~\ref{homogeneity},
\begin{equation}\label{concavity_3}
\mathcal{B}(P)\geq (1-2^{-n})\mathcal{B}(R_1)+2^{-n}(\cma{v^+}+\mathcal{B}(0,1)).
\end{equation}

But $\mathcal{B}$ is locally concave on $\Omega_*,$ so
$$\mathcal{B}(R_1)\geq \frac{|R_1-R|}{|P-R|}\,\mathcal{B}(P)+\frac{|P-R_1|}{|P-R|}\,\mathcal{B}(R)=\frac{|R_1-R|}{|P-R|}\,\mathcal{B}(P).$$
Plugging this into \eqref{concavity_3} and using the estimate $\mathcal{B}(0,1)\geq 1$ established in the previous lemma, we obtain an inequality equivalent to
$$ \mathcal{B}(P)\geq s^2(1+\cma{v^+})=\frac12\,s(1+s^2)=B(P),$$
where the last equality is just formula~\eqref{d_1} with $x_1=v.$ 

For $P\in \Omega_2$ the reasoning is similar \cma{and even simpler. We can take $R=R_1$ on $\Gamma_0,$ $R=\big(v-2^{-n/2},(v-2^{-n/2})^2\big);$ then $Q=\big(v+\tau,\left(v+\tau\right)^2+1\big)$ will be the point of intersection of $\Gamma_1$ with the line passing through $P$ and $R.$ Furthermore, the point $P$ splits the segment $[Q,R]$ in the right proportion $P=(1-2^{-n})R+2^{-n}Q.$ Therefore, we can use Lemma~\ref{mi} with $L=L^-=0$ and $L^+=v+\tau,$ since $v+\tau\ge0$:}
\begin{align*}
 \mathcal{B}(P)&\geq (1-2^{-n})\mathcal{B}\big(v-2^{-n/2},(v-2^{-n/2})^2\big)+
 2^{-n}\bel{B}^n\big(v+\tau,\left(v+\tau\right)^2+1, v+\tau\big)\\
 &=2^{-n}\bel{B}^n\big(v+\tau,\left(v+\tau\right)^2+1, v+\tau\big).
 \end{align*}
Now Lemma~\ref{homogeneity} and Lemma~\ref{x_1=0} with $y=0$ give
 $$ \mathcal{B}(P)\geq 2^{-n}\left(v+\tau+\mathcal{B}(0,1) \right)\geq 2^{-n}\left(v+\tau+1\right)
 =B(P),$$
 where we again used the fact that $\mathcal{B}(0,1)\ge1,$ and the last equality is simply formula~\eqref{d_2} with $x_1=v$ and $\alpha=2^{-n}.$ 
 
 Finally, assume that $P\in\Omega_{2k+1}\cup\Omega_{2k+2}$ for some $k\ge1.$ Then $T_{-j\tau}P\in\Omega_*$ for all $j\le k$ and $T_{-k\tau}P\in\Omega_1\cup\Omega_2.$ In light of the identity
$$
T_{-j\tau}P=(1-2^{-n})(v-j\tau-2^{-n/2},(v-j\tau-2^{-n/2})^2)+2^{-n}T_{-(j+1)\tau}P,
$$
a repeated application of Lemma~\ref{mi} with $L^\pm=L=0,$ along with Lemma~\ref{bc}, gives
\begin{align*}
\mathcal{B}(P)&\ge (1-2^{-n})\mathcal{B}(v-2^{-n/2},(v-2^{-n/2})^2)+2^{-n}\mathcal{B}(T_{-\tau}P)\\
&=2^{-n}\mathcal{B}(T_{-\tau}P)\ge\dots\ge 2^{-kn}\mathcal{B}(T_{-k\tau}P)\ge 2^{-kn}B(T_{-k\tau}P)= B(P),
\end{align*}
where the last equality comes from formula~\eqref{d_3+}. \cma{We have considered all applicable cases and the proof is complete}.
\end{proof}

We are now in a position to prove Lemma~\ref{l=0}.

\begin{proof}[Proof of Lemma~\ref{l=0}]
Every interior point of $x\in\Omega_*$ lies on a line segment $m_x$ contained entirely in $\Omega_*,$ connecting a point $U_x\in\{(u,u^2)\colon u\le0\}$ and a point $V_x\in\{(v,v^2+1)\colon v\le0\}\cup\{(0,y)\colon 0\le y\le 1\},$ and such that $B$ is linear along $m_x.$ Indeed, in $\Omega_0,$ $m_x$ is the horizontal segment $\{(t,x_2)\colon -\sqrt{x_2}\le t\le0\};$ in $\Omega_1$ and $\Omega_2,$ $m_x=\ell_{s(x)}$ given by~\eqref{vu_1} and~\eqref{vu_2}, respectively; and in the rest of $\Omega_*,$ $m_x$ is the image of a segment from $\Omega_1\cup\Omega_2$ under the transformation~\eqref{d_3+}, i.e., for an appropriate $k,$ $m_x=T_{k\tau}\ell_{s(\tilde x)},$ where $\tilde x=T_{-k\tau}x.$

Take $x\in\Omega_*.$ If $x$ is on the boundary of $\Omega_*,$ then the statement of the lemma follows from either Remark~\ref{b=0}, Lemma~\ref{x_1=0}, or Lemma~\ref{gamma_1}. If $x$ is in the interior of $\Omega_*,$ then we can write $x=(1-\gamma)U_x+\gamma V_x,$ \cma{where $U_x$ and $V_x$ are the endpoints of the corresponding segment $m_x.$} Then
$$
\mathcal{B}(x)\ge (1-\gamma) \mathcal{B}(U_x)+\gamma\mathcal{B}(V_x)=\gamma \mathcal{B}(V_x)\ge\gamma B(V_x)=(1-\gamma)B(U_x)+\gamma B(V_x)=B(x).\qedhere
$$
\end{proof}

\subsection{Norm-optimizing sequence}
We give the optimizing sequence for the candidate $B$ at the point $(0,1).$ As noted in the beginning of this section, this sequence is not needed to prove Lemma~\ref{BB}. However, its structure seems to us to be of interest, as it reflects the dual nature of the extremal problem we are solving. To explain: the Bellman formulation~\eqref{b_test}--\eqref{b_main} is \cma{similar to} the original $L^2$-formulation for the dyadic maximal function from~\cite{nt}, except the test functions are restricted to have BMO norm no more than 1. Accordingly, the optimizer we give below can be seen both as a special rearrangement of the dyadic logarithm from~\cite{alpha_trees}, designed to \cma{maximize that logarithm's} BMO norm, and as a close relative of the $L^2$-optimizer for the classical dyadic maximal operator constructed in~\cite{ssv} and~\cite{cincy}. 

\cma{
To construct the desired optimizing sequence, we first fix a positive integer $j$ and define on the interval $I_0:=(0,1)$ an auxiliary function $\psi$ using the following recursive formula:
\eq[psi]{
\psi(t)=
\begin{cases}
-\gamma,& 0<t\le 2^{-j},\\
\psi(2^kt-1),& 2^{-k}<t\le 2^{-k+1},~1<k\le j,\\
\psi(2t-1)+\delta,&\frac12<t<1.
\end{cases}
}
Let us verify that this formula uniquely defines $\psi$ almost everywhere for any fixed parameters $\gamma$ and $\delta.$
To that end, let us inductively define a sequence $\{\psi^{(m)}\}$ on some subset of $I_0:$
\eq[jm1]{
\psi^{(1)}(t)=\begin{cases}
-\gamma,& 0<t\le 2^{-j},
\\
{\rm not~defined},&2^{-j}<t<1;
\end{cases}
}
\eq[jm2]{
\psi^{(m)}(t)=\begin{cases}
-\gamma,& 0<t\le 2^{-j},\\
\psi^{(m-1)}(2^kt-1),& 2^{-k}<t\le 2^{-k+1},~1<k\le j,\\
\psi^{(m-1)}(2t-1)+\delta,&\frac12<t<1.
\end{cases}
}
Observe that the measure of the set where the function $\psi^{(m)}$ is not defined has measure $(1-2^{-j})^m,$ while at the points where it is defined, it will not change the value at the next step.
Therefore, this sequence converges to a function $\psi$ defined almost everywhere on $I_0$ and satisfying the recursive relation~\eqref{psi}.

The same argument proves the uniqueness of this function. Indeed, the difference of any two such functions would satisfy the same relation with $\gamma=\delta=0.$ Arguing as we did with the sequence $\{\psi^{(m)}\},$ we can guarantee that ``on the first step'' this difference equals zero on a set of measure no less than $2^{-j}.$ After using~\eqref{psi} $m$ times, we see that this difference equals zero on the set of measure no less than $1-(1-2^{-j})^m.$ Hence, it is zero almost everywhere.

Now, for any index $j\ge1$ let us choose the parameters $\gamma_j$ and $\delta_j$ so that the solution of~\eqref{psi} has average $0$ on $I_0$ and so that its square has average 1. It is easy to compute that we have to take 
\eq[gd]{
\gamma_j=\frac1{\sqrt{1+2^{1-j}}},\qquad \delta_j=\frac{2^{1-j}}{\sqrt{1+2^{1-j}}}.
}
Let $\psi_j$ be the function given by~\eqref{psi} with $\gamma=\gamma_j$ and $\delta=\delta_j$ and extended outside $I_0$ by zero. The one-dimensional maximal function $N_1\psi_j$ on $I_0$ will be determined only by the values of $\psi_j$ on $I_0$ and, hence, it will satisfy the same relation~\eqref{psi} with $\gamma=0$ and $\delta=\delta_j.$ Since the solution of~\eqref{psi} is unique, we conclude that
\eq[Mpsi]{
N_1\psi_j(t)=\psi_j(t)+\gamma_j.
}

The following lemma summarizes the key properties of the sequence $\{\psi_j\},$ some of which have already been noted. We leave its purely computational proof as an exercise for the reader. 
\begin{lemma}
\label{lemma_psi}
The sequence $\{\psi_j\}$ satisfies 
\begin{align*}
\forall j, &\quad \psi_j\in\BMO^d(I_0),\quad \|\psi_j\|_{\BMO^d(I_0)}=1;\\
\forall j, &\quad \av{\psi_j}{I_0}=0, \quad \av{\psi_j^2}{I_0}=1;\\
\forall j, &\quad \inf_{I_0}N_1\psi_j=0;\\
&\av{N_1\psi_j}{I_0}\longrightarrow 1,~\text{as}~j\to\infty.
\end{align*}
\end{lemma}

To extend the one-dimensional sequence $\{\psi_j\}$ to higher dimensions, 
let 
\eq[eta_j]{
\varphi_j(t_1,t_2,...,t_n)=
\psi_j(t_1).
}
The sequence $\{\varphi_j\}$ is immediately seen to be an optimizing sequence for the Bellman candidate $B$ at the point $(0,1)$ in the sense of Definition~\ref{d_opt}. Indeed, both the inclusion $\varphi_j\in F_{(0,1)}$ and the condition~\eqref{lim} follow at once from Lemma~\ref{lemma_psi}. 

Lastly, we fulfill the promise made in the introduction and provide a norm-optimizing sequence for the classical dyadic operator $M.$ 
Note that the sequence $\{\varphi_j+\gamma_j\}$ is non-negative on $\rn;$ therefore, on this sequence the operators $N$ and $M$ coincide. Furthermore, 
$$
\|\varphi_j+\gamma_j\|_{\BMO^d(\rn)}=\|\varphi_j\|_{\BMO^d(\rn)}=1
$$ 
and 
$$
\inf_{Q_0}M(\varphi_j+\gamma_j)=\inf_{Q_0}N(\varphi_j+\gamma_j)=\gamma_j.
$$
Therefore,
$$
\av{M(\varphi_j+\gamma_j)}{Q_0}-\inf_{Q_0}M(\varphi_j+\gamma_j)=\av{N(\varphi_j+\gamma_j)}{Q_0}-\gamma_j=\av{N\varphi_j}{Q_0}\to1,
$$
which means that $\|M\|_{\BMO\to\BLO}\ge 1.$}
\section{$B$ is $\alpha$-concave}
\label{alpha_concavity}
In this section, we prove Lemma~\ref{a-conc}, i.e., establish the fact that the function $B$ defined by~\eqref{domains}--\eqref{d_3+} is $\alpha$-concave on $\Omega.$ To that end, we will use a lemma from~\cite{alpha_trees} giving sufficient conditions for $\alpha$-concavity. Specifically, Lemma~2.5 of that paper contains the following statement (up to a slight change in notation).
\begin{lemma}
\label{L2}
Let $\alpha\in(0,\frac12].$
Assume that a function $B$ on $\Omega$ satisfies the following three conditions\textup:
\ben[label=(\arabic*),font=\upshape]
\item
\label{C1}
$B$ is locally concave on $\Omega\cma{.}$ 
\item
\label{C2}
$B$ has non-tangential derivatives at every point \cma{of} $\Gamma_1.$ Furthermore, for any two distinct points on $\Gamma_1,$  $P=(p,p^2+1)$ and $Q=(q,q^2+1)$ with $|p-q|\le\tau,$
\begin{align*}
(D_{\scriptscriptstyle\overrightarrow{PQ}}B)(P)&\ge (D_{\scriptscriptstyle\overrightarrow{PQ}}B)(Q),
\end{align*}
where $D_{_{\overrightarrow{PQ}}}$ denotes the derivative in the direction of the vector $\overrightarrow{PQ}.$
\item
\label{C3}
For any $P$ and $Q$ as above\textup, and $R\cma{:=}\frac1{1-\alpha}(P-\alpha Q),$
\begin{align*}
B(P)&\ge (1-\alpha)\, B(R)+\alpha\, B(Q).
\end{align*}
\een
Then $B$ is $\alpha$-concave on $\Omega.$
\end{lemma}

\begin{remark}
\label{period}
Recall that for $k\ge3$ and $x\in\Omega_k,$ we have
$$
B(x)=\alpha^{k}\,B(T_{k\tau}x).
$$
This fact will reduce the proof that $B$ satisfies Conditions~\ref{C2} and~\ref{C3} of Lemma~\ref{L2} to the consideration of the first several domains $\Omega_k$ and a few other special cases. For example, assume $P,Q\in\Omega_-$ and let $m^*$ be the largest integer $m$ such that $T_{-m\tau}P,T_{-m\tau}Q\in\Omega_-.$ Let $P^*=T_{-m^*\tau}P,$ $Q^*=T_{-m^*\tau}Q.$ Then 
$$
(D_{\scriptscriptstyle\overrightarrow{PQ}}B)(P)\ge (D_{\scriptscriptstyle\overrightarrow{PQ}}B)(Q)
\quad\Longleftrightarrow\quad 
(D_{\scriptscriptstyle\overrightarrow{P^*Q^*}}B)(P^*)\ge (D_{\scriptscriptstyle\overrightarrow{P^*Q^*}}B)(Q^*).
$$
A similar reduction applies to Condition~\ref{C3}.
\end{remark}
\cma{
\begin{remark}
Note that for the points $P$ $Q$ satisfying Condition~\ref{C2}, the point $R$ given in Condition~\ref{C3} always lies in $\Omega.$ Indeed, it is easy to show that $r_2-r_1^2=1-\frac{(p-q)^2}{\tau^2},$ thus, $r_2\ge r_1^2$ if and only if $|p-q|\le\tau.$
\end{remark}
}
We now verify the three conditions of Lemma~\ref{L2} as three separate lemmas.

\subsection{The proof of Condition \ref{C1} of Lemma~\ref{L2}}
\begin{lemma}
\label{lc}
$B$ is locally concave on $\Omega.$
\end{lemma}
\begin{proof}

It is easy to see that $B$ is continuous on $\Omega.$ Elementary differentiation shows that $B$ is locally concave in $\Omega_0$ and $\Omega_+$ and continuously differentiable in the interior of $\Omega_0\cup\Omega_+.$ Hence, it is locally concave in $\Omega_0\cup\Omega_+.$ Let us show that these properties extend to $\Omega_-.$

Observe that we can rewrite formulas~\eqref{vu_1}--\eqref{s_3+} in a uniform way, as follows. If $m\ge1$ and $x\in\Omega_m,$ let 
\cma{
$$
k=\Big[\frac m2\Big],\quad  z=\frac{s}{\alpha^k},\quad\mu=k\tau+1.
$$
}
In this notation, $u=\frac12\,(z-\frac1z)-\mu$ and $v=u+z$ if $m$ is odd, and $v=u+\frac1z$ if $m$ is even. Then the slope of the extremal segment $\ell_s$ is 
\cma{
$$
\frac{v^2+1-u^2}{v-u}=2u+v-u+\frac1{v-u}=2u+z+\frac1z=2(z-\mu).
$$ 
}
Therefore, $z$ is given as a function of $x=(x_1,x_2)$ by the equation
$
x_2=2(z-\mu)(x_1-u)+u^2
$
or, upon rewriting,
\eq[zx]{
x_2=2(z-\mu)x_1-\frac34\,z^2+2\mu z+\frac12-\mu^2+\frac1{4z^2}.
}
(\cma{The requirement $s\in[\alpha^{m/2},\alpha^{(m-1)/2}]$ determines the solution of this equation uniquely}.)
The formula for $B$ becomes
\eq[bzx]{
B(x)=\frac{\alpha^k}2\,(1+z^2)(x_1-u).
}
We compute the derivatives $z_{x_1}$ and $z_{x_2}$ from~\eqref{zx}:
$$
z_{x_1}=\frac{2(\mu-z)}{2x_1-\frac32\,z+2\mu-\frac1{2z^3}},\qquad 
z_{x_2}=\frac{1}{2x_1-\frac32\,z+2\mu-\frac1{2z^3}}.
$$
Observe that we have $x_1\le v.$ If $m$ is odd, then $z\le1\le \mu$ and $v=\frac12\,(3z-\frac1z)-\mu,$ thus $\mu-z\ge0$ and
$$
2x_1-\frac32\,z+2\mu-\frac1{2z^3}\le \frac32\,z-\frac1z-\frac1{2z^3}=\frac1{2z^3}\,(3z^2+1)(z^2-1)\le0.
$$
If $m$ is even, then $z\le\frac1{\sqrt\alpha}$ and $v=\frac12\,(z+\frac1z)-\mu,$ thus $\mu-z\ge\tau+1-\frac1{\sqrt\alpha}=1-\sqrt\alpha\ge0$ and
$$
2x_1-\frac32\,z+2\mu-\frac1{2z^3}\le -\frac12\,z+\frac1z-\frac1{2z^3}=-\frac1{2z^3}\,(z^2-1)^2\le0.
$$
Hence, in all cases we have $z_{x_1}\le0,$ $z_{x_2}\le0.$ Furthermore,~\eqref{bzx} yields
\begin{align*}
\alpha^{-k}B_{x_1}&=\frac{1}2\,(1+z^2)+\Big(z(x_1-u)-\frac12\,(1+z^2)u_z\Big)z_{x_1}\\
&=\frac{1}2\,(1+z^2)+\frac{z}{2}\,\Big(2x_1-\frac32\,z+2\mu-\frac1{2z^3}\Big)z_{x_1}\\
&=\frac{1}2\,(1+z^2)+z(\mu-z)=\frac{1}2\,(1-z^2)+z\mu.
\end{align*}
Equivalently,
\eq[z1]{
\alpha^{-k}B_{x_1}=-uz.
}
Similarly,
\eq[z2]{
\alpha^{-k}B_{x_2}=\Big(z(x_1-u)-\frac12\,(1+z^2)u_z\Big)z_{x_2}=\frac z2.
}
Therefore,
$$
\alpha^{-k}B_{x_1x_1}=(\mu-z)z_{x_1},\quad \alpha^{-k}B_{x_1x_2}=(\mu-z)z_{x_2},\qquad \alpha^{-k}B_{x_2x_2}=\frac12\,z_{x_2},
$$
which gives
\eq[ma1]{
B_{x_1x_1}\le0,\quad B_{x_2x_2}\le0,\quad B_{x_1x_1}B_{x_2x_2}=B_{x_1x_2}^2.
}
This means that $B$ is locally concave in each $\Omega_m.$ Furthermore, since $B_{x_1}=-us,$ $B_{x_2}=\frac s2,$ and $s$ is a continuous function of $x$ on $\Omega_-,$ we conclude that $B\in C^1(\Omega_-).$ Hence, it is locally concave in $\Omega_-.$ 

It remains to consider the boundary between $\Omega_0$ and $\Omega_-,$ i.e., the extremal segment $\ell_1.$ As shown above, in $\Omega_-,$ $B_{x_1}=-us,$ $B_{x_2}=\frac s2.$ On $\ell_1,$ we have $B_{x_1}=1,$ $B_{x_2}=\frac 12.$ On the other hand, in $\Omega_0,$ $B_{x_1}=1$ and $B_{x_2}=\frac1{2\sqrt{x_2}}.$ On $\ell_1,$ we have $x_2=1,$ thus $B_{x_2}=\frac 12.$ Therefore, $\nabla B$ is continuous across $\ell_1$ and so $B$ is locally concave in all of $\Omega.$
\end{proof}

\subsection{Important formulas}
Let us collect in one place several key formulas we will need in the rest of this section.  \cma{First, we already know that in $\Omega_-$
\eq[ders]{
B_{x_1}=-us,\qquad B_{x_2}=\frac s2.
}
}
These are formulas~\eqref{z1} and~\eqref{z2}. The functions $u=u(s)$ and $s=s(x)$  are defined by~\eqref{vu_1},~\eqref{vu_2}, and~\eqref{s_3+}.

We will also need the expressions in terms of $s$ of $B$ and its tangential derivative on the upper boundary of $\Omega_-.$ For $v\le0,$ let $V=(v,v^2+1).$ Recall the notation $b(v)=B(V).$ From~\eqref{d_1},~\eqref{d_2}, and~\eqref{d_3+},
\eq[bb]{
b(v)=
\begin{cases}
\frac{s}2\,\big(1+\cma{\big(\frac{s}{\alpha^k}\big)^2}\big),&\text{if~}V\in\Omega_{2k+1},~k\ge0,\\
\frac{s}2\,\big(1+\cma{\big(\frac{\alpha^k}s\big)^2}\big),&\text{if~}V\in\Omega_{2k},~k\ge1.
\end{cases}
}

Furthermore,~\eqref{ders} gives
\eq[b'1]{
b'(v)=B_{x_1}(V)+2vB_{x_2}(V)=s(v-u).
}
Using~\eqref{vu_1},~\eqref{vu_2}, and~\eqref{s_3+}, this can be written as
\eq[b']{
b'(v)=
\begin{cases}
\frac{s^2}{\alpha^k},&\text{if~}V\in\Omega_{2k+1},~k\ge0,\\
\alpha^k,&\text{if~}V\in\Omega_{2k},~k\ge1.
\end{cases}
}

\subsection{The proof of Condition~\ref{C2} of Lemma~\ref{L2}}
\begin{lemma}
\label{dirder}
For any two points on $\Gamma_1,$ $P =(p, p^2+1)$ and $Q=(q,q^2+1),$ such that $|p-q|\le\tau,$ we have
\eq[dir_der]{
(D_{\scriptscriptstyle\overrightarrow{PQ}}B)(P)\ge (D_{\scriptscriptstyle\overrightarrow{PQ}}B)(Q).
}
\end{lemma}
\begin{proof}
It is enough to check~\eqref{dir_der} only when $p\le q,$ since the condition is symmetric. Moreover, since the function $x_1+\sqrt{x_2-x_1^2}$ is concave in the whole domain $\{x_1\ge0, x_2\ge x_1^2\},$ \eqref{dir_der} holds automatically when $0\le p\le q,$ hence, we can assume $p\le0.$ 

We have $\overrightarrow{PQ}=(q-p)
\big[\begin{smallmatrix}
1\\
p+q
\end{smallmatrix}\big],
$
thus inequality~\eqref{dir_der} is equivalent to $B_{x_1}(P)+B_{x_2}(P)(p+q)\ge B_{x_1}(Q)+B_{x_2}(Q)(p+q)$ or, using~\eqref{b'1}, to
\eq[refr]{
b'(p)-b'(q)+(q-p)(B_{x_2}(P)+B_{x_2}(Q))\ge0.
}

When $q\ge0,$ we have $b'(q)=1$ and $B_{x_2}(Q)=\frac12,$ which means that the left-hand side of~\eqref{refr} is an increasing function of $q,$ \cma{since $B_{x_2}(P)\ge0.$} Therefore, the case $q\ge0$ reduces to $q=0.$ 

From now on, assume that $q\le0.$ Then $B_{x_2}(P)=\frac{s_p}2$ and $B_{x_2}(Q)=\frac{s_q}2,$ the left-hand side of~\eqref{refr} is a function of $s_p$ and $s_q:$ 
$$
G(s_p,s_q):=b'(p)-b'(q)+\frac12\,(q-p)(s_p+s_q),
$$
and~\eqref{refr} is equivalent to the inequality $G\ge0.$

\cma{
We have the following possibilities for the location of $P$ and $Q$ on $\Gamma_1$: 
\begin{itemize}
\item$k\ge1,$ $P\in\Omega_{2k},$ $Q\in\Omega_{2k}\cup\Omega_{2k-1};$ 
\item$k\ge2,$ $P\in\Omega_{2k},$ $Q\in\Omega_{2k-2};$
\item$k\ge1,$ $P\in\Omega_{2k+1},$ $Q\in\Omega_{2k+1}\cup\Omega_{2k}\cup\Omega_{2k-1}.$ 
\end{itemize}
By Remark~\ref{period}, it is enough to consider the following three cases: 
\begin{itemize}
\item$P\in\Omega_2,$ $Q\in\Omega_2\cup\Omega_1;$
\item$P\in\Omega_4,$ $Q\in\Omega_2;$
\item$P\in\Omega_3,$ $Q\in\Omega_3\cup\Omega_2\cup\Omega_1.$
\end{itemize}
}

{\bf Case 1: $P\in\Omega_2,$ $Q\in\Omega_2\cup\Omega_1.$} 
When $P\in\Omega_2,$ we have $b'(p)=\alpha.$ If $Q\in\Omega_2,$ then $b'(q)=\alpha$ as well, so $G\ge0.$ If $Q\in\Omega_1,$ then $b'(q)=s_q^2,$ so~
$$
G(s_p,s_q)=\alpha-s_q^2+\frac{q-p}2\,(s_p+s_q),
$$
where $\alpha\le s_p\le \sqrt\alpha,$ $p=\frac12(\frac{s_p}\alpha+\frac{\alpha}{s_p})-\tau-1,$ $\sqrt\alpha\le s_q\le1,$ and $q=\frac12(3{s_q}-\frac{1}{s_q})-1.$ Direct differentiation shows that this function is concave in $s_q,$ thus it suffices to verify that 
$G(s_p,\sqrt\alpha)\ge0$ and $G(s_p,1)\ge0.$ The first inequality is obvious, and the second one is equivalent to
$$
g(s_p):=\alpha-1-\frac{p}2\,(s_p+1)\ge0\cma{.}
$$
The function $g$ is easily seen to be concave in $s_p,$ thus we need only verify that $g(\alpha)\ge0$ and $g(\sqrt\alpha)\ge0.$ For $s_p=\alpha,$ we have $p=-\tau,$ so $g(\alpha)=\frac1{2\sqrt\alpha}(1-\alpha)(1-\sqrt\alpha)^2\ge0.$ For $s_p=\sqrt\alpha,$ we have $p=\frac12(3\sqrt\alpha-\frac1{\sqrt\alpha})-1,$ so $g(\sqrt\alpha)=\frac1{4\sqrt\alpha}(1+\sqrt\alpha)(1-\sqrt\alpha)^2
\ge0.$ 

{\bf Case 2: $P\in\Omega_4,$ $Q\in\Omega_2.$} Here, $\alpha^2\le s_p\le \alpha^{3/2},$ $p=\frac12(\frac{s_p}{\alpha^2}+\frac{\alpha^2}{s_p})-2\tau-1,$ and $b'(p)=\alpha^2;$ $\alpha\le s_q\le \frac{s_p}\alpha,$ $q=\frac12(\frac{s_q}\alpha+\frac{\alpha}{s_q})-\tau-1,$ and $b'(q)=\alpha.$ Thus,
$$
G(s_p,s_q)=\alpha^2-\alpha+\frac{q-p}2\,(s_p+s_q).
$$
This function is increasing in $s_q$ since $q$ is. Thus, it suffices to verify that $G(s_p,\alpha)\ge0.$ Observe that when $s_q=\alpha,$ we have $Q\in\Omega_3.$ Since we have already proved that $G\ge0$ in the algebraically equivalent case $P\in\Omega_2,$ $Q\in\Omega_1,$ we conclude that $G(s_p,\alpha)\ge0.$

{\bf Case 3: $P\in\Omega_3,$ $Q\in\Omega_3\cup\Omega_2\cup\Omega_1.$} 

Here we have $\alpha^{3/2}\le s_p\le \alpha,$ $p=\frac12(\frac{3s_p}\alpha-\frac\alpha{s_p})-\tau-1,$ and $b'(p)=\frac{s_p^2}\alpha.$

If $Q\in\Omega_3,$ then $s_p\le s_q\le \alpha,$ $q=\frac12(\frac{3s_q}\alpha-\frac\alpha{s_q})-\tau-1,$ and $b'(q)=\frac{s_q^2}\alpha.$ Then
\begin{align*}
G(s_p,s_q)&=\frac1\alpha(s_p^2-s_q^2)+\frac14\Big(\frac3\alpha(s_q-s_p)-\alpha\Big(\frac1{s_q}-\frac1{s_p}\Big)\Big)(s_p+s_q)\\
&=\frac1{4\alpha s_ps_q}\,(s_q^2-s_p^2)(\alpha^2-s_ps_q)\ge0.
\end{align*}

If $Q\in\Omega_2,$ then $\alpha\le s_q\le \sqrt\alpha,$ $q=\frac12(\frac{s_q}\alpha+\frac\alpha{s_q})-\tau-1,$ and $b'(q)=\alpha.$ Then
\begin{align*}
G(s_p,s_q)&=\frac{s_p^2}\alpha-\alpha+\frac14\Big(\frac{s_q}\alpha-\frac{3s_p}\alpha+\frac\alpha{s_q}+\frac\alpha{s_p}\Big)(s_p+s_q)\\
&=\frac1{4\alpha }\,(s_q-s_p)^2+\frac\alpha4\Big(\frac{s_q}{s_p}+\frac{s_p}{s_q}-2\Big)\ge0.
\end{align*}

If $Q\in\Omega_1,$ then $\sqrt\alpha\le s_q\le \frac{s_p}\alpha,$ $q=\frac12(3s_q-\frac1{s_q})-1,$ and $b'(q)=s_q^2.$ Then 
\begin{align*}
G(s_p,s_q)=\frac{s_p^2}\alpha-s_q^2+\frac{q-p}2\,(s_p+s_q).
\end{align*}
This function is concave in $s_q,$ so it is enough to verify that $G(s_p,\sqrt\alpha)\ge0$ and $G(s_p,\frac{s_p}\alpha)\ge0.$ When $s_q=\sqrt\alpha,$ $Q\in\cma{\Omega_2},$ so this has been shown above. When $s_q=\frac{s_p}\alpha,$ we have $q=p+\tau,$ so
$$
G\Big(s_p,\frac{s_p}\alpha\Big)=\frac{s_p^2}\alpha-\frac{s_p^2}{\alpha^2}+\frac\tau2\,\Big(1+\frac1\alpha\Big)s_p.
$$
This function is concave in $s_p,$ so it suffices to verify that $G(\alpha^{3/2},\sqrt\alpha)\ge0$ and $G(\alpha,1)\ge0.$ We compute:
$$
G(\alpha^{3/2},\sqrt\alpha)=\frac12\,(1-\alpha)^2,\qquad G(\alpha,1)=\frac1{2\sqrt\alpha}(1-\alpha)(1-\sqrt\alpha)^2
\ge0.
$$
The proof is complete.
\end{proof}

\subsection{The proof of Condition \ref{C3} of Lemma~\ref{L2}}
For all $p,q$ such that $|q-p|\le\tau,$ let
$$
H(p,q)=b(p)-(1-\alpha)B(R)-\alpha b(q),
$$
where $R=(r_1,r_2)=\cma{\big(\frac{p-\alpha q}{1-\alpha},\frac{p^2-\alpha q^2}{1-\alpha}+1\big)}.$ Then Condition~\ref{C3} is equivalent to the inequality $H(p,q)\ge0.$

\begin{lemma}
\label{main_alpha}
For all $p$ \cma{and} $q$ such that $|q-p|\le\tau,$ $H(p,q)\ge0.$
\end{lemma}
\cma{The proof consists of a series of lemmas, each dealing with an important special case, followed by the remaining general case.}
\begin{lemma}
\label{p<q}
If $p-\tau\le q\le p,$ then $H(p,q)\ge0.$
\end{lemma}
\begin{proof}
Note that $H$ is everywhere continuously differentiable and
\begin{align*}
H_q&=-(1-\alpha)\Big(B_{\cma{x_1}}(R)\,\frac{\partial r_1}{\partial q}+B_{x_2}(R)\frac{\partial r_2}{\partial q}\Big)
-\alpha(B_{x_1}(Q)+2q\,B_{x_2}(Q))\\
&=\alpha\Big(B_{x_1}(R)+2qB_{x_2}(R)-B_{x_1}(Q)-2qB_{x_2}(Q)\Big).
\end{align*}
Furthermore, $H$ \cma{is everywhere second-differentiable in the distributional sense}, and its second derivative with respect to $q$ is given by 
\begin{align*}
H_{qq}=&-\alpha\bigg(\frac{\alpha}{1-\alpha}\,\Big[B_{x_1x_1}(R)+4qB_{x_1x_2}(R)+4q^2B_{x_2x_2}(R)\Big]\\
&+\Big[B_{x_1x_1}(Q)+4qB_{x_1x_2}(Q)+4q^2B_{x_2x_2}(Q)\Big]\bigg)+\cma{2}\alpha(B_{x_2}(R)-B_{x_2}(Q)).
\end{align*}
Since $B$ is locally concave, each of the two terms in square brackets is non-positive. Furthermore, we also have $B_{x_2}(R)\ge B_{x_2}(Q).$ Let us explain: if $R\in\Omega_-,$ then $Q\in\Omega_-,$ and $(v_r,v_r^2+1)$ is to the right of $R,$ thus to the right of $Q.$ Therefore, $B_{x_2}(R)=\frac{s_r}2\ge \frac{s_q}2=B_{x_2}(Q).$ If $R\in\Omega_0,$ then $B_{x_2}(R)=\frac1{2\sqrt{r_2}};$ if $R\in\Omega_+,$ then $B_{x_2}(R)=\frac1{2\sqrt{r_2-r_1^2}}.$ In either case, $B_{x_2}(R)\ge\frac12.$ On the other hand, if $Q\in\Omega_-,$ then $B_{x_2}(Q)=\frac{s_q}2\le \frac12,$ and if $Q\in\Omega_+,$ then $B_{x_2}(Q)=\frac12.$ 

Therefore, $H_{qq}\ge0$ \cma{and, hence}, $H_q$ is increasing in $q$ for $q\le p.$ Since $H_q(p,p)=0,$ we conclude that $H_q\le0.$ Therefore, the minimum of $H$ \cma{for $q\in[p-\tau\le p]$ is attained when $q=p.$} Since $H(p,p)=0,$ that minimum is 0.
\end{proof}
Thus, we may assume that $p\le q.$ Furthermore, observe that if $P\in\Omega_+,$ then the line segment $[R,Q]$ lies in the domain $\Omega_0\cup\{x_1\ge0, x_2\ge x_1^2\}.$ Since the function $x_1+\sqrt{x_2-x_1^2}$ is concave in \cma{the region $\{x_2\ge x_1^2\},$} $B$ allows a locally concave extension to 
$$
\Omega_0\cup\{x_1\ge0, x_2\ge x_1^2\},
$$ 
which implies that $H(p,q)\ge0.$ Therefore, we may assume that $\cma{P}\in\Omega_-.$

\begin{lemma}
\label{R0}
If $R\in\Omega_0,$ then $H(p,q)\ge0.$
\end{lemma}
\begin{proof}
For this case, we must have $P\in\Omega_1$ and $Q\in\Omega_+.$ Therefore,
$$
b(p)=\frac12\,s_p(s_p^2+1),\quad B(R)=r_1+\sqrt{r_2},\quad b(q)=q+1.
$$
Since $r_2\le 1,$ $B(R)\le r_1+1=\cma{\frac{p-\alpha q}{1-\alpha}}+1.$ Then,
\begin{align*}
H(p,q)&=\frac12\,s_p(s_p^2+1)-(1-\alpha)(r_1+\sqrt{r_2})-\alpha(q+1)\\
&\ge \frac12\,s_p(s_p^2+1)-(1-\alpha)\Big(\frac{p-\alpha q}{1-\alpha}+1\Big)-\alpha(q+1)\\
&= \frac12\,s_p(s_p^2+1)-p-1=\frac12\,s_p(s_p^2+1)-\frac12\Big(3s_p-\frac1{s_p}\Big)=\frac1{2s_p}\,(s_p^2-1)^2\ge0. \qedhere
\end{align*}
\end{proof}

From now on, assume that $R\in\Omega_-.$ Let $u,v,v^+$ be the horizontal coordinates of three points on the extremal trajectory passing through $R:$ the point of intersection with $\Gamma_0,$ the left point of intersection with $\Gamma_1,$ and the right point of intersection with $\Gamma_1,$ respectively (the last two points may coincide).  
\cma{Since $B$ is linear along the trajectory, we have}
$$
B(R)=\frac{r_1-u}{v-u}\,b(v).
$$
Write $\xi=v-u,$ $\theta=\frac{r_1-u}{v-u},$ and $\delta=q-p.$ Then on one hand we have
$$
r_2-r_1^2=(1-\theta)u^2+\theta(v^2+1)-((1-\theta)u+ \theta v)^2=\theta+\theta(1-\theta)\xi^2,
$$
while on the other hand,
$$
r_2-r_1^2=\frac{p^2+1-\alpha(q^2+1)}{1-\alpha}-\Big(\frac{p-\alpha q}{1-\alpha}\Big)^2=1-\frac{\delta^2}{\tau^2}.
$$
Thus,
\eq[delta]{
\delta=\tau\sqrt{(1-\theta)(1-\xi^2\theta)}.
}
In addition, since $r_1=\theta\xi+u=v-(1-\theta)\xi,$
\eq[lmv]{
p=v-(1-\theta)\xi+\frac{\alpha\delta}{1-\alpha},~\qquad q=v-(1-\theta)\xi+\frac\delta{1-\alpha}.
}
\cma{We will also need the expression for $v^+$ in terms of $v$ and $\xi.$ Since
$$
\frac{(v^+)^2-v^2}{v^+-v}=\frac{v^2+1-u^2}{v-u}=v+u+\frac1{v-u},
$$
we have
\eq[v+]{
 v^+=v+\frac1\xi-\xi.
}
}
It is easy to see that there are only two possibilities for the order of the numbers $p,q,v,v^+:$ either $v\le p\le q\le v^+$ or $p\le v\le v^+\le q.$ There are two key special cases that we address first: $p=v\le v^+=q$ and $p\le v=v^+\le q.$

\begin{lemma}
\label{trajectory}
If $p=v,$ and $q=v^+,$ then $H(p,q)\ge0.$
\end{lemma}
\begin{proof}
Let $V=(v,v^2+1),$ $V^+=(v^+,(v^+)^2+1).$ \cma{We have} $r_1=\frac{v-\alpha v^+}{1-\alpha}$ \cma{and}
$$
\frac{r_1-u}{v-u}=\frac{\frac{v-\alpha (v+1/\xi-\xi)}{1-\alpha}-(v-\xi)}\xi=\frac1{1-\alpha}\,\Big(1-\cma{\frac\alpha{\xi^2}}\Big).
$$
\cma{Therefore,}
$$
H(p,q)=H(v,v^+)=b(v)-(1-\alpha)\,\frac{r_1-u}{v-u}\, b(v)-\alpha b(v^+)=\alpha\Big(\frac{b(v)}{\xi^2}-b(v^+)\Big),
$$
and to prove the lemma is the same as to show that
\eq[vv^+]{
\frac{b(v)}{\xi^2}-b(v^+)\ge0.
}
From now on, the specifics of the consideration will depend on the location of the point $V.$ In accordance with Remark~\ref{period}, it is enough to consider two cases: $V\in\Omega_2$ and $V\in\Omega_1.$ 

The case $V\in\Omega_1$ is the easier one and we will start with it. In this case, $v=\frac12(3\xi-\frac1\xi)-1$ and $b(v)=\frac12(\xi^3+\xi).$ Furthermore, $v^+=v+\frac1\xi-\xi=\frac12(\xi+\frac1\xi)-1.$ Therefore, $v^+\ge0$ and $b(v^+)=v^++1=\frac12(\xi+\frac1\xi).$ Then
$$
\frac{b(v)}{\xi^2}-b(v^+)=\frac{\frac12(\xi^3+\xi)}{\xi^2}-\frac12\Big(\xi+\frac1\xi\Big)=0.
$$

If $V\in\Omega_2,$ then $v=\frac12(\xi+\frac1\xi)-\tau-1$ and $b(v)=\frac\alpha2(\xi+\frac1\xi).$ Furthermore, $v^+=v+\frac1\xi-\xi=\frac12(\frac3\xi-\xi)-\tau-1.$ The value of $b(v^+)$ is determined by the location of the point $V^+,$ and there are three possibilities: $V^+\in\Omega_2,$ $V^+\in\Omega_1,$ and $V^+\in\Omega_+.$ 

If $V^+\in\Omega_2,$ then $b(v^+)=\alpha(v^++\tau+1)=\frac\alpha2(\frac3\xi-\xi).$ Thus,
$$
\frac{b(v)}{\xi^2}-b(v^+)=\frac{\frac\alpha2(\xi+\frac1\xi)}{\xi^2}-\frac\alpha2\Big(\frac3\xi-\xi\Big)=
\frac\alpha{2\xi^3}\,\big(1-\xi^2\big)^2\ge0.
$$

If $V^+\in\Omega_+,$ then $b(v^+)=v^++1=\frac12(\frac3\xi-\xi)-\tau.$ Thus,
\begin{align*}
\frac{b(v)}{\xi^2}-b(v^+)&=\frac{\frac\alpha2(\xi+\frac1\xi)}{\xi^2}-\frac12\Big(\frac3\xi-\xi\Big)+\tau
=\frac1{2\xi^3}\,\big(\xi^4+2\tau\xi^3+(\alpha-3)\xi^2+\alpha\big)\\
&=\frac1{2\xi^3}\,\big(\xi-\sqrt\alpha\big)^2\Big(\xi^2+\frac2{\sqrt\alpha}\,\xi+1\Big)\ge0.
\end{align*}

Finally, if $V^+\in\Omega_1,$ then $v^+=\frac12(3z-\frac1z)-1$ and $b(v^+)=\frac12(z^3+z),$ for some $z\in[\sqrt\alpha,1].$ \cma{Since $v^+=\frac12(\frac3\xi-\xi)-\tau-1,$ we have}
\eq[zxi]{
3z-\frac1z=\frac3\xi-\xi-2\tau.
}
\cma{Hence,}
$$
\frac{b(v)}{\xi^2}-b(v^+)=\frac12\,\Big[\alpha\Big(\frac1{\xi^3}+\frac1\xi\Big)-(z^3+z)\Big].
$$
Let us turn things around: fix $z\in(0,1]$ and consider the last expression in brackets as a function of $\alpha,$ i.e., let
$S(\alpha)=\alpha\big(\frac1{\xi^3}+\frac1\xi\big)-(z^3+z),$ where $0<\alpha\le z^2$ and $\xi=\xi(\alpha)$ is the positive solution of~\eqref{zxi}.
We have $\frac{d\xi}{d \alpha}=\frac{\alpha^{-3/2}+\alpha^{-1/2}}{3/\xi^2+1},$ thus
$$
S'(\alpha)=\alpha\Big(-\frac3{\xi^4}-\frac1{\xi^2}\Big)\frac{d\xi}{d \alpha}+\frac1{\xi^3}+\frac1\xi=-\frac\alpha{\xi^2}\,(\alpha^{-3/2}+\alpha^{-1/2})+\frac1{\xi^3}+\frac1\xi=\frac1{\xi^3}\,(\xi-\sqrt\alpha)\Big(\xi-\frac1{\sqrt\alpha}\Big)\le0,
$$
\cma{since in $\Omega_1$ we have $v-u=s\in[\sqrt\alpha,1]$ and, thus, $\sqrt\alpha\le\xi\le1\le\frac1{\sqrt\alpha}.$}
Therefore, to show that $S(\alpha)\ge0,$ it suffices to show that $S(z^2)\ge0.$ From~\eqref{zxi}, when $\alpha=z^2,$ we have $z+\frac1z=\frac3{\xi}-\xi,$ hence, writing $\xi$ for $\xi(z^2),$
$$
S(z^2)=z^2\Big(\frac1{\xi^3}+\frac1\xi\Big)-(z^3+z)=z^2\Big[\frac1{\xi^3}+\frac1\xi-\Big(\frac3{\xi}-\xi\Big)\Big]=\frac{z^2}{\xi^3}\,(1-\xi^2)^2\ge0.
$$
The proof is complete.
\end{proof}

\begin{lemma}
\label{tangent}
If $\cma{v^+=v},$ and $p\le v\le q,$ then $H(p,q)\ge0.$
\end{lemma}
\begin{proof}
The condition $\cma{v^+=v}$ means that the extremal trajectory is tangent to the curve $\Gamma_1.$ Such tangent trajectories all have $v=-k\tau$ for some \cma{integer} $k\ge1$ and serve as the boundaries between $\Omega_{2k}$ and $\Omega_{2k+1}.$ For our purposes, it suffices to consider $v=-\tau,$ which corresponds to the boundary between $\Omega_2$ and 
$\Omega_3.$

We have $\xi=1,$ thus from~\eqref{delta}, $\theta=\frac{r_1-u}{v-u}=1-\frac{\delta}\tau.$ Since $b(v)=\alpha,$ the inequality $H(p,q)\ge0$ can be rewritten as
$$
D(\delta):=b(p)-(1-\alpha-\sqrt\alpha\delta)\alpha-\alpha b(q)\ge0,
$$
where, from~\eqref{lmv},
$$
p=v-\frac{\sqrt\alpha \delta}{1+\sqrt\alpha},\quad  q=v+\frac{\delta}{1+\sqrt\alpha}.
$$
The domain of the function $D$ is $[0,\tau].$ Clearly, $D(0)=0$ and $D(\tau)=0.$

We have
$$
D'(\delta)=-\frac{\sqrt\alpha}{1+\sqrt\alpha}\,b'(p)+\cma{\alpha\sqrt\alpha}-\alpha b'(q)\,\frac1{1+\sqrt\alpha}.
$$
Setting $D'(\delta)$ equal to 0, and using the fact that $b'(v)=\alpha,$ we get
\eq[derz]{
\frac1{1+\sqrt\alpha}\,b'(p)+\frac{\sqrt\alpha}{1+\sqrt\alpha}\,b'(q)=b'(v).
}
To prove the lemma, we will first show that $D'(\tau)<0$ and then that equation~\eqref{derz} has no more than one root $\delta^*$ in the interval $(0,\tau).$ Since $D(0)=0,$ this will imply that there is precisely one root $\delta^*,$ which is a point of local maximum. Hence, the minimum of $D$ on $[0,\tau]$ is attained at the endpoints, thus being 0.

To show that $D'(\tau)<0,$ we note that this is equivalent to the inequality
$$
\frac1{1+\sqrt\alpha}\,b'(p_*)+\frac{\sqrt\alpha}{1+\sqrt\alpha}\,b'(p_*+\tau)>\alpha,
$$
where $p_*=v-\frac{\sqrt\alpha \tau}{1+\sqrt\alpha}=\cma{-\tau}-1+\sqrt\alpha.$ Since $\cma{b'(p_*+\tau)=\frac{b'(p_*)}\alpha},$ this can be written simply as
\eq[p*]{
b'(p_*)>\alpha^{3/2}.
}
Since $(p_*,p_*^2+1)\in\cma{\Omega_3},$ we have $p_*=\frac12(3x_*-\frac1{x_*})-\tau-1$ 
\cma{for some $x_*\in[\sqrt\alpha,1].$ Comparing the two expressions for $p_*$ we obtain the equation $3x_*-\frac1{x_*}=2\sqrt\alpha;$ in the interval $[\sqrt\alpha,1]$ it has the unique root 
$x_*=\frac13(\sqrt\alpha+\sqrt{\alpha+3}).$ Since $b'(p_*)=\alpha x_*^2,$~\eqref{p*} becomes $x_*^2>\sqrt\alpha,$ i.e.,
$$
\sqrt\alpha+\sqrt{\alpha+3}>3\alpha^{1/4},
$$}
which in turn can be easily seen to be a true statement for all $\alpha\in(0,1).$

We now show that there is no more than one root of equation~\eqref{derz} inside $(0,\tau).$ The key observation is that the function $b'$ is strictly convex on the interval $[p_2,-\tau]$ and convex on the interval $[-\tau,0].$ Indeed, for $t\in[p_2,-\tau],$ we have $b'(t)=\alpha x^2,$ 
\cma{where $x$ is the unique solution of the equation $t=\frac12(3x-\frac1x)-\tau-1$ that lies in $[\sqrt\alpha,1]$ ($x=\frac s\alpha$ from~\eqref{b'}).} Viewed as a function of $t,$ $x$ is increasing and strictly convex, thus so is $b'.$ For $t\in[-\tau,p_1],$ $b'(t)=\alpha$ and for $t\in[p_1,0],$ $b'$ is again convex by the same argument as for $[p_2,-\tau].$ Since $b'$ is increasing, it is convex on $[-\tau,0].$

Now, assume that $0<\delta_1<\delta_2<\tau$ are two roots of equation~\eqref{derz}. Then $q(\delta_2)>q(\delta_1),$ which, by the convexity of $b'$ to the right of $v,$ implies that
$$
\frac{b'(q(\delta_2))-b'(v)}{q(\delta_2)-v}\ge \frac{b'(q(\delta_1))-b'(v)}{q(\delta_1)-v}.
$$ 
By~\eqref{derz}, $\frac{b'(q)-b'(v)}{q-v}=\frac{b'(v)-b'(p)}{v-p},$ thus, 
$$
\frac{b'(v)-b'(p(\delta_2))}{v-p(\delta_2)}\ge \frac{b'(v)-b'(p(\delta_1))}{v-p(\delta_1)}.
$$ 
However, we also have $p(\delta_2)<p(\delta_1),$ which, by the strict convexity of $b'$ to the left of $v,$ implies that
$$
\frac{b'(v)-b'(p(\delta_2))}{v-p(\delta_2)}<\frac{b'(v)-b'(p(\delta_1))}{v-p(\delta_1)}.
$$ 
This contradiction proves that there is at most one root of equation~\eqref{derz} in $(0,\tau).$  (As noted above, there is in fact precisely one such root; see Figure~\ref{derz_pic}.) This completes the proof.
\end{proof}
\begin{figure}[h!]
\centering{
\includegraphics[width=13cm]{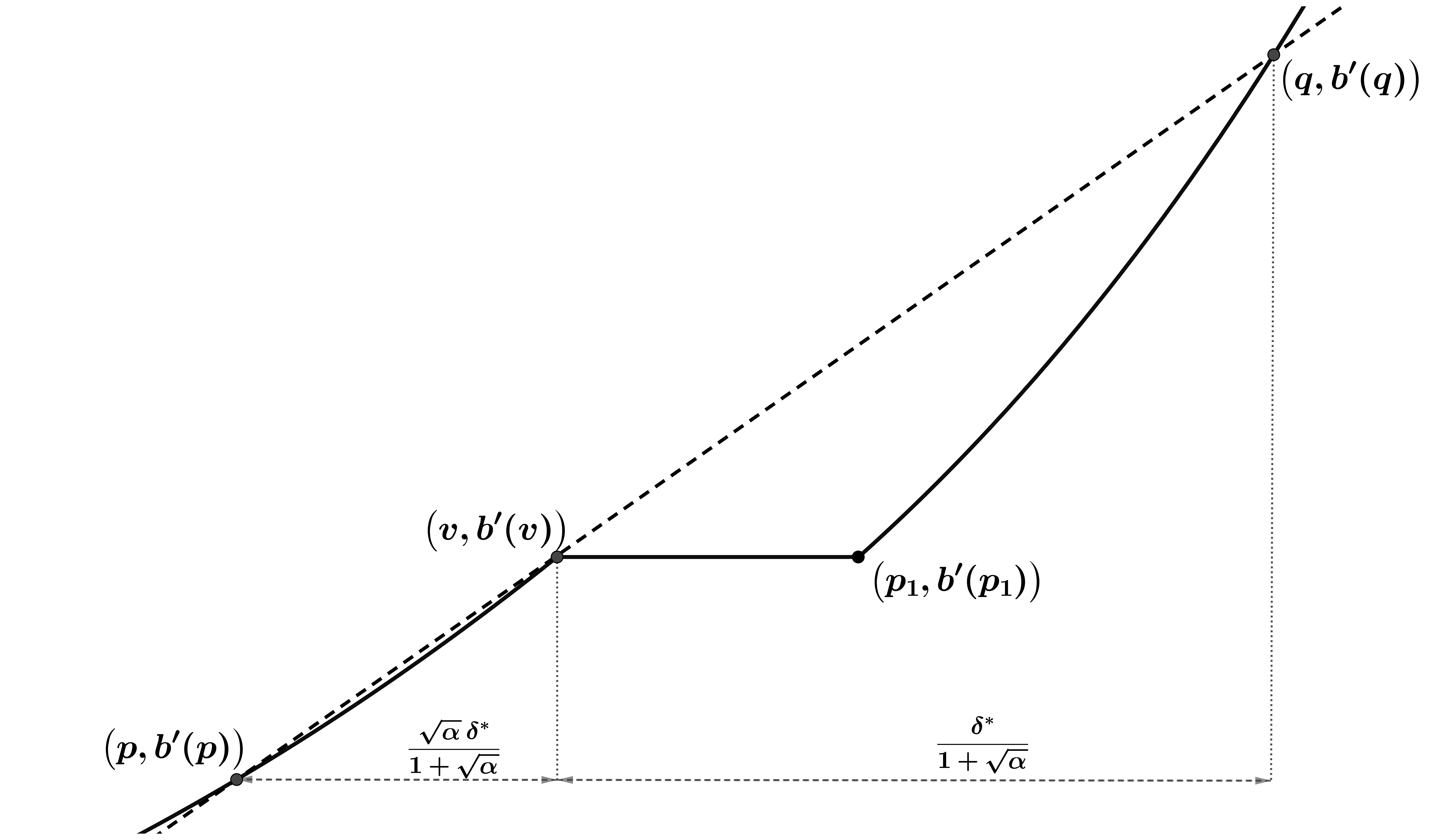}
\caption{The root $\delta^*$ of the equation $D'(\delta)=0$}
\label{derz_pic}
}
\end{figure}

\cma{
To simplify further calculations, we need to consider one more special case. \cbl{Recall definition~\eqref{tau} of the numbers $p_k:$ $p_0=\frac12\sqrt\alpha+\frac1{2\sqrt\alpha}-1,$ $p_k=p_0-k\tau.$
\begin{lemma}
\label{extension}
If $k\ge1,$ $p\in[p_k,-k\tau],$ and $q\in[p,p_{k-1}],$ then $H(p,q)\ge0.$
\end{lemma}
}
\begin{proof}
The proof relies on the fact that for such $p$ and $q$ there exists a function $\tilde B$ that coincides with $B$ at $P,$ $Q,$ and $R$ and that is locally concave in a domain that contains the segment $[R,Q].$ \cbl{In light of Remark~\ref{period} it is enough to consider the case $k=1.$}

Recall the family of extremal line segments $\{\ell_s\}_{\sqrt\alpha\le s\le 1}$ connecting the points $(u(s),u^2(s))$ and $(v(s),v^2(s)+1)$ with $u(s)$ and $v(s)$ given by~\eqref{vu_1}. Each $x\in\Omega_1$ lies on exactly one such segment and $B(x)$ is given by~\eqref{d_1}.
Now, for each $s$ let $\tilde\ell_s$ be the extension of $\ell_s$ until the second point of intersection with $\Gamma_1.$ Thus, $\tilde\ell_s$ connects the points $(u(s), u^2(s))$ and $(v^+(s), (v^+(s))^2+1),$ where $v^+(s)$ is given by~\eqref{v+} with $\xi=s$: $v^+(s)=v(s)+\frac1s-s.$ Let $\omega_1$ be the region lying under the $\tilde\ell_{\sqrt\alpha}$ and above $\Gamma_1:$ 
$$
\omega_1=\{x\colon: p_1\le x_1\le p_0,~x_1^2+1\le x_2\le (p_1+p_0)x_1-p_1p_0+1\}
$$ 
Let $\tilde\Omega_1=\Omega_1\cup\omega_1;$ see Figure~\ref{tilde_omega_1}. Then each point $x\in\tilde\Omega_1$ lies on exactly one segment $\tilde\ell_s.$ To define $\tilde B(x),$ we simply extend definition~\eqref{d_1} to $\tilde\Omega_1:$
$$
\tilde B(x)=\frac12\,(1+s^2)(x_1-u).
$$
\begin{figure}[h]
\centering{
\includegraphics[width=17cm]{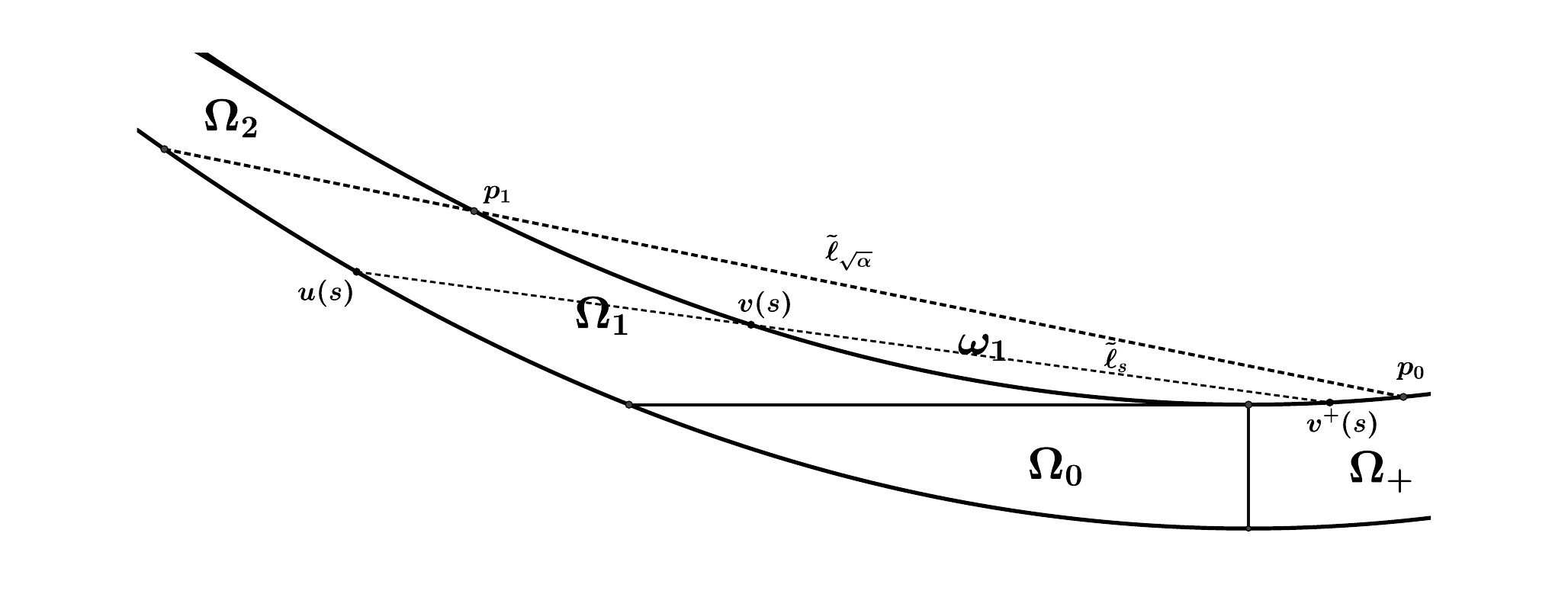}
\caption{The region $\tilde\Omega_1=\Omega_1\cup\omega_1$ along with a generic segment $\tilde\ell_s$}
\label{tilde_omega_1}
}
\end{figure}

Observe that $\tilde B(x)=B(x)$ for $x\in\Gamma_1\cap\{0\le p_0\}.$ Furthermore, the argument in~Lemma~\ref{lc} goes through without any changes and we conclude that $\tilde{B}$ is locally concave in~$\tilde\Omega_1.$ 

Let $\tilde{B}(x)=B(x)$ for $x\in\Omega_2;$ then $\tilde{B}$ is locally concave on $\Omega_2\cup\tilde\Omega_1.$ It remains to observe that if $p$ and $q$ are as in the statement of the lemma, then $[R,Q]\in\Omega_2\cup\tilde\Omega_1,$ which means that
$$
H(p,q)=B(P)-(1-\alpha)B(R)-\alpha B(Q)=\tilde B(P)-(1-\alpha)\tilde B(R)-\alpha \tilde B(Q)\ge0. \qedhere
$$
\end{proof}
\begin{remark}
The proof of Lemma~\ref{main_alpha} would have been much shorter if a similar locally concave extension covering all applicable segments $[R,Q]$ were available for $\Omega_2.$ Unfortunately, this is not the case: the maximal domain of extension is bounded by the envelope of the extremal segments corresponding to $\Omega_2;$ since this envelope is both convex and external to $\Omega$ (see Figure~\ref{f3} in the next section), this maximal domain will not be sufficient for our purposes.
\end{remark}
}
\cma{We are now in a position to finish the proof of Lemma~\ref{main_alpha}}. 
\begin{proof}[Proof of the general case]

We parametrize all applicable triples $R,P,Q$ in Condition~\ref{C3} of Lemma~\ref{L2} by the location of the extremal trajectory corresponding to $R$ and the location of $R$ within that trajectory. Take a number $\xi\in[\sqrt\alpha,1].$ Setting $v-u=\xi$ and specifying the domain $\Omega_k$ where $R$ lies uniquely determines $v.$
Let $V(\xi)=(v,v^2+1)$ and $U(\xi)=(u,u^2).$ 

Now, let $\theta\in[0,1]$ be such that $R=(1-\theta)U(\xi)+ \theta V(\xi).$ This in turn defines, \cma{as functions of $\xi$ and $\theta,$} two points $P(\xi,\theta)=(p,p^2+1)$ and $Q(\xi,\theta)=(q,q^2+1)$ such that $P=(1-\alpha)R+\alpha Q$ and $p\le q.$ 

Using~\eqref{lmv}, we see that to show that $H(p,q)\ge0$ for all pairs $p,q$ such the point $R=\frac{P-\alpha Q}{1-\alpha}$ lies on the extremal line connecting $U$ and $V,$ is the same as to show that the function
\eq[A(xi,theta)]{
W(\xi,\theta):=b\Big(v-(1-\theta)\xi+\frac{\alpha\delta}{1-\alpha}\Big)-(1-\alpha)\theta b(v)-\alpha b\Big(v-(1-\theta)\xi+\frac\delta{1-\alpha}\Big)
}
is non-negative on the domain $\{\sqrt\alpha\le\xi\le1,~0\le\theta\le 1\}.$ 

Let us first consider the boundary of this domain. If $\theta=0,$ then $\delta=\tau$ and we have
$W=b(p)-\alpha b(p+\tau)=0.$
If $\theta=1,$ then $R=P=Q$ and $W=0.$ If $\xi=\sqrt\alpha,$ \cbl{then the trajectory connecting $U$ and $V$ separates $\Omega_{2k-1}$ and $\Omega_{2k}$ for some $k\ge1.$ Therefore, $v=p_k$ and $v^+=v+\tau=p_{k-1},$ where the numbers~$p_k$ are defined by~\eqref{tau} and~$v^+$ is given by~\eqref{v+}. Since $q-p\le\tau,$ we have $v\le p\le q\le v^+,$ so by Lemma~\ref{extension} $W\ge0.$}
 Finally, if $\xi=1,$ then $W\ge0$ by Lemma~\ref{tangent}.
  
We now show that $W$ does not have non-negative extrema inside the domain. The partial derivatives are
$$
W_\theta=b'(p)\Big(\xi+\frac{\alpha\delta_\theta}{1-\alpha}\Big)-(1-\alpha)b(v)-\alpha b'(q)\Big(\xi+\frac{\delta_\theta}{1-\alpha}\Big)
$$
and
$$
W_\xi=b'(p)\Big(v_\xi-1+\theta+\frac{\alpha\delta_\xi}{1-\alpha}\Big)-(1-\alpha)\cma{\theta}b'(v)v_\xi-
\alpha b'(q)\Big(v_\xi-1+\theta+\frac{\delta_\xi}{1-\alpha}\Big).
$$
Setting $W_\theta$ and $W_\xi$ equal to $0$ and rearranging gives the following equation for $b'(p)$:
\eq[eq_p]{
b'(p)\big[\delta_\xi\xi-(v_\xi-1+\theta)\delta_\theta\big]=(1-\alpha)\left(b(v)\Big(v_\xi-1+\theta+\frac{\delta_\xi}{1-\alpha}\Big)
-\theta b'(v)v_\xi\Big(\xi+\frac{\delta_\theta}{1-\alpha}\Big)\right).
}
From~\eqref{delta}, we have $\delta_\theta=\frac{\tau^2}{\delta}\,(\xi^2\theta-\frac{1+\xi^2}2)$ and 
$\delta_\xi=-\frac{\tau^2}\delta\,\xi\theta(1-\theta).$ Furthermore, as can be seen from~\eqref{bb} and~\eqref{b'}, we also have $b'(v)=\frac{2\xi}{\xi^2+1}\,b(v).$ After plugging these expressions into~\eqref{eq_p} and simplifying, we see that the left-hand side becomes
$$
-b'(p)\,\frac{\tau^2}{2\delta}\,\big[(1+\xi^2)(1-\theta)+v_\xi(2\xi^2\theta-1-\xi^2)\big],
$$ 
while the right-hand side becomes
$$
-b'(v)\,\frac{1-\alpha}{2\xi}\,\Big(\frac{1-\alpha}{\alpha\delta}\,\theta\xi+1\Big)\big[(1+\xi^2)(1-\theta)+v_\xi(2\xi^2\theta-1-\xi^2)\big].
$$
It is easy to show that for any location of the point $V,$ the common factor in these expressions is never zero, unless $\theta=1$ and $\xi=1,$ in which case, of course, $W=0.$ Assuming that is not the case, after cancellation and rearrangement~\eqref{eq_p} becomes
\eq[pv']{
b'(p)=b'(v)\Big(1-\frac{v-p}\xi\Big).
}
We would now like to consider all possible locations of the point $R$ or, equivalently, the point $V.$ According to Remark~\ref{period}, it is enough to consider \cma{two cases: $V\in\Omega_1$ and $V\in\Omega_2.$

Assume first that $V\in\Omega_1.$ That means that $v=\frac12(3\xi-\frac1\xi)-1$ and $b'(v)=\xi^2.$  We must also have $P\in\Omega_1$ (otherwise, $q>p+\tau$); thus, $p=\frac12(3z-\frac1z)-1$ for some $z\in[\sqrt\alpha,1]$ and $b'(p)=z^2.$ Then~\eqref{pv'} becomes
$$
z^2=\xi^2\left(1-\frac{\frac12(3\xi-\frac1\xi)-\frac12(3z-\frac1z)}\xi\right)\quad\implies\quad\xi^2-\xi\Big(3z-\frac1z\Big)+2z^2-1=0.
$$
This gives either $\xi=z$ or $\xi=2z-\frac1z.$ In the first case, $p=v,$ thus $W\ge0$ by Lemma~\ref{trajectory}. In the second case, $\xi\le z,$ thus $p\ge v\ge p_1,$ which means that $q\le v^+\le p_0$ and, thus, $W\ge0$ by Lemma~\ref{extension}.

Assume now that $V\in\Omega_2.$ That means that $v=\frac12(\xi+\frac1\xi)-\tau-1$ and $b'(v)=\alpha.$ It is clear from the geometry that we must have either $P\in\Omega_1,$ $P\in\Omega_2,$ or $P\in\Omega_3.$ If $P\in\Omega_1,$ Lemma~\ref{extension} applies and we have $W\ge0.$ If $P\in\Omega_2,$ then
then $b'(p)=\alpha$ and~\eqref{pv'} gives $p=v,$ thus $W\ge0$ by Lemma~\ref{trajectory}. If $P\in\Omega_3,$ then $p=\frac12(3z-\frac1z)-\tau-1$ for some $z\in[\sqrt\alpha,1]$ and $b'(p)=\alpha z^2.$ Then~\eqref{pv'} becomes
$$
z^2=1-\frac{\frac12(\xi+\frac1\xi)-\frac12(3z-\frac1z)}\xi\quad\implies\quad\frac1{\xi^2}-\frac1\xi\Big(3z-\frac1z\Big)+2z^2-1=0.
$$
Solving for $\frac1\xi$ we have either $\frac1\xi=z$ or $\frac1\xi=2z-\frac1z.$ Since $\cma{z\le1}$ and $\xi\le1,$  
the only possible solution in either case is $\xi=1$ and $z=1.$ That means that the points $R,$ $P,$ and $Q$ coincide and $W=0.$
}
\end{proof}

\section{How to find the Bellman candidate $B$}
\label{how_to}
Recall the notation: $\Omega_-=\cup_{k\ge1}\Omega_k,$ $\Omega_*=\cup_{k\ge0}\Omega_k.$
It is a straightforward matter to find the Bellman candidate $B$ in the domain $\Omega_0,$ which can be seen to be the maximal convex part of $\Omega_*$ that includes all of the boundary $x_1=0.$ Specifically, using the arguments from~\cite{ssv} and~\cite{cincy}, we seek the function $A(x;L)=L+B(T_Lx)$ on $S$ that satisfies the homogeneous \ma equation in $x,$ $A_{x_1x_1}A_{x_2x_2}=A_{x_1x_2}^2,$ as well as the boundary condition $\frac{\partial A}{\partial L}|_{x_1=L}=0.$ When translated to $B,$ these requirements yield the function $B(x)=x_1+\sqrt{x_2}$ in $\Omega_0.$ Since we also want $B$ to satisfy Condition~\ref{cond2} of Lemma~\ref{induction} (with $L=0$), this also means that $B(x)=x_1+\sqrt{x_2-x_1^2}$ in $\Omega_+.$ 

To construct $B$ in $\Omega_-,$ we first compute it on the upper boundary and then solve a certain \ma boundary value problem. To find the formula for $b(x_1):=B(x_1,x_1^2),$ we use an idea from~\cite{melas}. In that paper, Melas found the Bellman function for the dyadic maximal operator on $L^2(\rn).$  His Bellman function -- let us call it simply $\bel{B}$ here --  also had the variables $x_1, x_2,$ and $L,$ defined the same way as in~\eqref{b_test}. He first found the function $\bel{B}\cma{(x,x_1)}$ and then used it to find the full formula for $\bel{B}\cma{(x,L)}.$ We will employ a somewhat similar reasoning here, though things are significantly complicated by the fact that $\Omega_-$ is non-convex, due to the BMO restriction $x_2\le x_1^2+1,$ absent in~\cite{melas}.

\subsection{The candidate in $\Omega_1$}
Using a variant of Melas's procedure, we are looking for a function $b(x_1)=B(x_1,x_1^2+1)$ as
\eq[bsup]{
b(x_1)=\sup\{(1-s^2)L+s^2(y_1+1)\}=\sup\{s^2(y_1+1)\}\,,
}
where $\sup$ is taken over all $s\in[0,1]$ and all points $y=(y_1,y_1^2+1)$ \cma{
for which there is a $z\in\Omega$ (i.e., $z=(z_1, z_2)$ and $0\le z_2-z_1^2\le 1$) such that $x=(x_1,x_1^2+1)$ is a convex combination of $y$ and $z$: $x=(1-s^2)z+s^2y.$}
The domain for the variable $y_1$ is determined from the condition
$$
z_2=\frac{x_1^2-s^2y_1^2}{1-s^2}+1\ge\Big(\frac{x_1-s^2y_1}{1-s^2}\Big)^2=z_1^2\,,
$$
which gives
$
\cma{|y_1-x_1|}\le\frac{1-s^2}s\,.
$
Thus, the supremum in~\eqref{bsup} is attained for $y_1=x_1+\frac{1-s^2}s:$
$$
b(x_1)=\sup_{0\le s\le1}\big\{s-s^3+(1+x_1)s^2\big\}\,.
$$
Let us write $\eta=1+x_1$. Since $x_1\le0$, we have $\eta\le1$.
We are looking for the maximal value of the cubic polynomial $-s^3+\eta s^2+s$
on the interval $s\in[0,1]$. Its derivative $-3s^2+2\eta s+1$ has two roots of
different signs. The positive root,
\eq[root]{
s=\frac{\eta+\sqrt{\eta^2+3}}3,
}
lies in the interval $[0,1].$
Hence, the supremum in the definition of $b$ is attained for this specific $s;$ from now on, 
the letter $s$ will denote not a free parameter, but the function of $x_1$ defined
by~\eqref{root}, and
\eq[H(s)]{
b(x_1)=s+\eta s^2-s^3=\frac12s(s^2+1)=:h(s)\,.
}

Having determined the function $b,$ we now aim to find the minimal concave function $B$ in $\Omega_-$ subject to two boundary conditions: $B|_{\Gamma_1}=b$ and $B|_{\Gamma_0}=0.$  The graph of any such function is a ruled surface; thus, $\Omega_-$ is foliated by straight-line segments along which $B$ is linear \cma{and its gradient is constant (we have earlier called such segments extremals)}. Let the points $(u,u^2)$ and $(v,v^2+1)$
be the two endpoints of \cma{an extremal}. That means that the tangent
vectors to the boundary curves of the graph of $B,$ along with the line passing through
the points $(u,u^2,B(u,u^2))=(u,u^2,0)$ and $(v,v^2+1,b(v)),$ all
lie in the same plane. Therefore,
$$
\det
\begin{pmatrix}
1 & 2u & 0
\\
1 & 2v & b'(v)
\\
v-u & v^2+1-u^2 & b(v)
\end{pmatrix}=0\,.
$$
This gives the following equation:
\eq[v-u]{
(v-u)^2-2\frac{b(v)}{b'(v)}(v-u)+1=0\,,
}
which explicitly determines the first coordinate $u$ of the endpoint on $\Gamma_0$ as a function of the first coordinate $v$ of the endpoint on $\Gamma_1.$
Let us solve this equation for our specific $b$.
Since
$
\frac{ds}{d\eta}=\frac{2s^2}{3s^2+1}
$
and $\frac{d\eta}{dx_1}=1,$ we have
$$
b'(v)=h'(s)\,\frac{2s^2}{3s^2+1}=\frac12(3s^2+1)\cdot\frac{2s^2}{3s^2+1}=s^2\quad\implies\quad 
\frac{b(v)}{b'(v)}=\frac{s^2+1}{2s},
$$
meaning equation~\eqref{v-u} has two roots: $v-u=s$ and $v-u=\frac1s$. From the geometry
it is clear that $|v-u|\le1$, so we have to take the first root, $v-u=s$. Hence,
$$
u=v-s\,.
$$

It is convenient to parametrize our extremals by $s\in(0,1]$. Then the
extremal $\ell_s$ is the segment connecting the points $(v,v^2+1)$ and $(u,u^2),$ where
\eq[v(s)]{
v=v(s)=\frac12\,\Big(3s-\frac1s\Big)-1,\qquad u=u(s)=v(s)-s=\frac12\,\Big(s-\frac1s\Big)-1.
}
Thus, the slope of $\ell_s$ is
$
\frac{v^2+1-u^2}{v-u}=-2(1-s),
$
and the equation of $\ell_s$ is
\eq[extrem]{
\begin{aligned}
x_2&=-2(1-s)(x_1-u)+u^2=-2(1-s)x_1+\frac1{4s^2}-\frac12+2s-\frac34s^2\,.
\end{aligned}
}
Finally, for $x\in\Omega_-,$ we let 
$$
B(x)=\frac{x_1-u}{v-u}\,h(s)=\frac12\,(1+s^2)(x_1-u),
$$
where $s=s(x)$ is given by~\eqref{extrem}.

Note that the second root of~\eqref{v-u}, which is $v-u=\frac1s,$ corresponds to the second point of intersection
the line containing $\ell_s$ with the upper parabola. Denoting this point $(v^+,(v^+)^2+1),$ we have
\eq[v_plus]{
v^+(s)=u(s)+\frac1s=\frac12\,\Big(s+\frac1s\Big)-1\,.
}
Subject to a verification of its properties, we have constructed the minimal locally concave function in $\Omega_-$ with specified boundary values on $\Gamma_0\cap\Omega_-$ and $\Gamma_1\cap\Omega_-.$ Let us recall that the ultimate goal is the construction of an $\alpha$-concave function on $\Omega.$ Such a function must satisfy inequality~\eqref{601} for all $\beta\cma{\in}[\alpha,\frac12].$ However, we do not want to make it too accommodating, in the sense of satisfying this definition for even smaller $\beta.$ In other words, in considering the segments $[x^-,x^+]$ as in Definition~\ref{def}, we want to ensure that the portion of the $[x^-,x^+]$ that lies outside of $\Omega$ is no larger, relative to the whole segment, than $\cma{1-\alpha}.$ When applied to the points $x^-=(u,u^2)$ and $x^+=(v^+,(v^+)^2+1),$ this  requirement gives
$$
\frac{v-u}{v^+-u}\ge\alpha\qquad\Longleftrightarrow\qquad \cma{s^2}\ge\alpha.
$$
Therefore, we will restrict $s$ to the interval $[\sqrt\alpha,1]$ in our construction. The smallest $v$ is then $v=p_1=\frac32\,\sqrt\alpha-\frac1{2\sqrt\alpha}-1$ (cf.~\eqref{tau}), meaning we have defined our candidate $B$ in $\Omega_1.$   

We would like to find the envelope of the family $\{\ell_s\}$. Let $x(s)=(x_1(s),x_2(s))$ be the 
tangent point of $\ell_s$ to the envelope. Then the slope of $\ell_s$ equals the derivative
\eq[slope]{
\frac{dx_2}{dx_1}=\frac{x'_2(s)}{x'_1(s)}=-2(1-s)\,.
}
Hence, the graph of $x(s)$ is a concave curve starting with the zero slope at the point $(0,1)$ and having the limit slope of $-2$. Differentiating~\eqref{extrem} with $x_i=x_i(s)$ we get
$$
x'_2(s)=-2(1-s)x'_1(s)+2x_1(s)-\frac1{2s^3}+2-\frac32\,s\,.
$$
\cma{Using~\eqref{slope} to solve this for $x_1$ and then using~\eqref{extrem} to get $x_2,$ we have}
$$
x_1(s)=\frac1{4s^3}-1+\frac34s=\frac{(1-s)^2(1+2s+3s^2)}{4s^3}\,,\quad x_2(s)=-\frac{2-3s-6s^3+6s^4-3s^5}{4s^3}\,.
$$
The extremal trajectories in $\Omega_1$ along with their envelope are shown in Figure~\ref{f3}.

\subsection{The candidate in the rest of $\Omega_-$}
Having constructed the extremal foliation for candidate $B$ in the domain $\Omega_1,$ we need to understand the foliation to the left of $\ell_{\sqrt\alpha}$. Again, we first determine $B$ on the upper parabola. The basic idea behind our definition comes from previous work on dyadic BMO, most importantly~\cite{alpha_trees}: we postulate that the main $\alpha$-concavity inequality~\eqref{601} (with $F=B$) becomes an equality when when $\beta=\alpha,$ $x^-$ is on $\Gamma_0,$ and both $x^+$ and $(1-\alpha)x^-+\alpha x^+$ are on $\Gamma_1.$ This choice is geometrically intuitive, as this configuration maximizes the portion of the segment $[x^-,x^+]$ that is external to the domain. 

For all $v\le p_1,$ we let 
\eq[rec]{
b(v)=(1-\alpha)B((v-\sqrt\alpha),(v-\sqrt\alpha)^2)+\alpha b(v+\tau)=\alpha b(v+\tau)\,.
}
We now seek the smallest locally concave function in $\Omega_-\setminus\Omega_1$ with the specified boundary conditions on the upper and lower boundaries. To that end, we need to determine the foliation of the domain by extremal segments.

Let us first consider the interval 
$v\in\big[-\tau,p_1\big]$,
when $v+\tau\ge0$ and $b(v+\tau)=v+\tau+1$. For such $v$ we have
\eq[h(v)1]{
b(v)=\alpha\big(v+\tau+1\big)\,.
}
It will be convenient to parametrize our trajectories $\ell_s$ by a parameter $s$ in such a way that $v-u=\frac\alpha s$.
After plugging this in~\eqref{v-u} together with
$$
\frac{b(v)}{b'(v)}=v+\tau+1
$$
we get
$$
s^2-2v\alpha s-2(1+\sqrt\alpha-\alpha)\sqrt\alpha s+\alpha^2=0\,.
$$
Hence,
\eq[v(s)1]{
v=\frac12\,\Big(\frac{s}\alpha+\frac{\alpha}s\Big)-\tau-1\,,\qquad 
u=v-\frac\alpha s=\frac12\,\Big(\frac{s}\alpha-\frac{\alpha}s\Big)-\tau-1.
}
When $s$ decreases from $\sqrt\alpha$ to $\alpha$, the value of $v$
deceases from 
$p_1$ to $-\tau$.
In terms of $s$ we can rewrite the expression~\eqref{h(v)1} for $b$:
\eq[H(s)1]{
b(v(s))=\frac{\cma{\alpha^2}}{2s}\,\Big(1+\frac{s^2}{\alpha^2}\Big)\,=:h(s)
}
The extremal segment $\ell_s$ has the equation
\eq[extrem-t]{
x_2=2\Big(\frac{s}{\alpha}-\tau-1\Big)x_1
-\frac{3s^2}{4\alpha^2} +2(\tau+1)\frac s\alpha+\frac12-(\tau+1)^2+\frac{\alpha^2}{4s^2}.
}
As before, the function $B$ is linear on the extremal line and can be calculated using its values at the
ends of $\ell_s$:
$$
B(x)=\frac{x_1-u}{v-u}\,h(s)=\frac\alpha{2}\,\Big(1+\frac{s^2}{\alpha^2}\Big)\,(x_1-u),
$$
where $s=s(x)$ is defined by~\eqref{extrem-t} \cma{and $u=u(s)$ is defined by~\eqref{v(s)1}}.
\begin{figure}[h]
\centering{
\includegraphics[width=16cm]{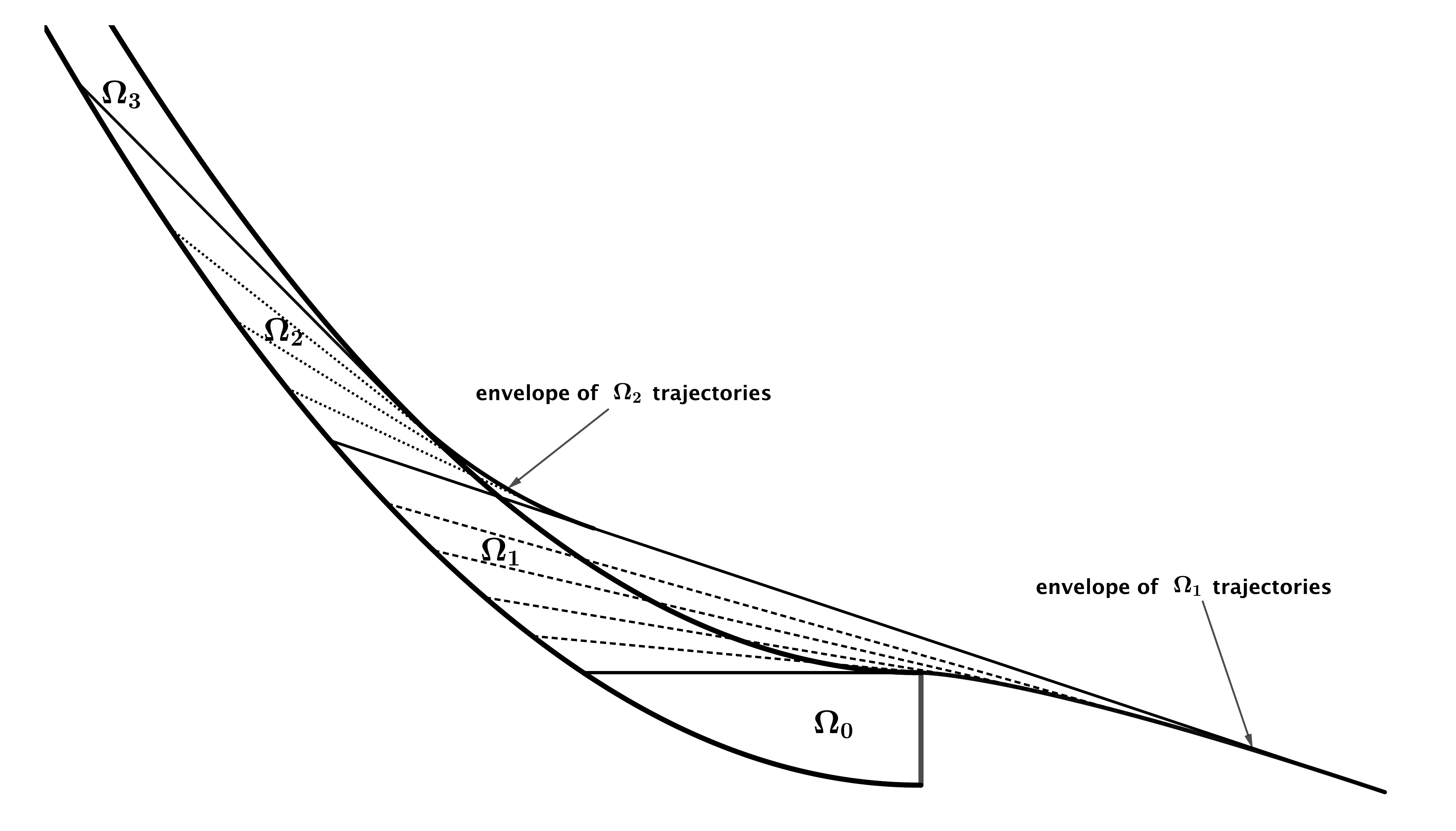}
\caption{The extremal trajectories in $\Omega_1$ and $\Omega_2$ and their envelopes}
\label{f3}
}
\end{figure}
Now, let us find the envelope of the family $\{\ell_s\}$. As before, let $x(s)=(x_1(s),x_2(s))$ be
the coordinate of the tangent point of $\ell_s$ to the envelope. Proceeding as in the previous case, we obtain
$$
x_1(s)=\frac{3s}{4\alpha}+\frac\alpha{4s}-\tau-1,\qquad 
x_2(s)=(\tau+1)^2
-\frac{3s^4+\alpha^4}{2s^3\alpha}\,(\tau+1)+
\frac{3s^4+2s^2\alpha^2+3\alpha^4}{4s^2\alpha^2}\,.
$$
The extremal trajectories in $\Omega_2$ along with their envelope are shown in Figure~\ref{f3}.

We have constructed the Bellman candidate $B$ in $\Omega_1$ and $\Omega_2.$ In the rest of $\Omega_-,$ formula~\eqref{rec} dictates that the foliation be determined by~\eqref{s_3+} and, thus, that $B$ be defined by~\eqref{d_3+}.

\end{document}